\documentclass[11pt]{article}
\usepackage{amsfonts, amssymb, amsmath, amsthm, amssymb}
\usepackage{esint}
\usepackage{pifont}
\usepackage{bbding}
\usepackage{latexsym}
\usepackage{mathrsfs}
\usepackage{graphicx}
\usepackage{lineno}
\usepackage[colorlinks, linkcolor=blue, anchorcolor=blue, citecolor=blue]{hyperref}
\usepackage{verbatim}
\allowdisplaybreaks 
\usepackage[left=1in, right=1in, bottom=1.4in]{geometry}

\newtheorem{theorem}{Theorem}[section]
\newtheorem{lemma}[theorem]{Lemma}

\newtheorem{remark}[theorem]{Remark}

\newcommand{\R}{\mathbb{R}}

\newcommand{\be}{\beta}

\newcommand{\va}{\varepsilon}
\newcommand{\al}{\alpha}
\newcommand{\pa}{\partial}

\newcommand{\lm}{\lambda}

\newcommand{\De}{\Delta}
\newcommand{\na}{\nabla}

\newcommand{\per}{\text{per}}

\date{} 

\numberwithin{equation}{section}

\begin{document}

\title{Combined Effects of Homogenization and Singular Perturbations: Quantitative Estimates}

\author{Weisheng Niu\thanks{Supported by the NSF of China (11971031, 11701002).}\qquad 
     Zhongwei Shen\thanks{Supported in part by NSF grant DMS-1856235.}}

\maketitle
\pagestyle{plain}
\begin{abstract}
We investigate quantitative estimates in periodic homogenization of 
second-order elliptic systems of elasticity with singular fourth-order perturbations.
 The convergence rates, which depend on the scale $\kappa$ that represents the strength of the singular perturbation 
 and on the length scale $\va$ of the heterogeneities, are established. 
 We also obtain the large-scale Lipschitz estimate, down to the  scale $\va$  and independent of $\kappa$.
 This large-scale estimate, when combined with small-scale estimates, yields  the  classical Lipschitz estimate 
 that is uniform in both $\va$ and $\kappa$.
 
 \medskip
 


\end{abstract}


\section{Introduction}\label{section-1}

In this paper we aim  to quantify the combined effects of homogenization and singular perturbations for the elliptic system,
\begin{align}\label{eq1}
\mathcal{L}_\varepsilon ( u_\varepsilon ) =F   \quad\text{ in } \Omega,
\end{align}
where $\Omega \subset  \mathbb{R}^d$ $  (d\geq 2)$ is  a bounded domain and 
\begin{align}
\label{lva}
\mathcal{L}_\varepsilon=\kappa^2\Delta^{2}-\textrm{div}(A(x/\varepsilon)\nabla),\quad\, 0<\varepsilon, \kappa  <1.
\end{align}
 The coefficient matrix (tensor)  $A(y)=( a^{\al\be}_{ij}(y))$, with $ 1\le  \alpha, \beta, i, j\le d$, is assumed 
  to be real, bounded measurable and to satisfy the elasticity condition,
\begin{equation}\label{econ}
\aligned
& a_{ij}^{\alpha\beta} (y) =a_{ji}^{\beta \alpha} (y) = a_{\alpha j}^{i \beta} (y),\\
 & \nu_1 |\xi|^2 \le a_{ij}^{\alpha\beta} \xi_i^\alpha \xi_j^\beta \le \nu_2 |\xi|^2
 \endaligned
\end{equation}
for a.e.~$y\in \mathbb{R}^d$ and for any symmetric matrix $\xi=(\xi_i^\alpha)\in \mathbb{R}^{d\times d}$,
where $\nu_1, \nu_2$ are positive constants.
We also assume  that $A$ is 1-periodic; i.e.,
\begin{align}\label{pcon}
A(y+z)=A(y) ~\textrm{ for any } z\in \mathbb{Z}^{d} \textrm{ and a.e. } y\in \mathbb{R}^{d}.
\end{align}

The elliptic operator in (\ref{lva}) arises in the study of the formation of the so-called shear bands in elastic materials subject to
severe loadings \cite{francfort1994}.
Variational functionals associated with the related  nonlinear operators are also used to model the heterogeneous thin films of martensitic materials \cite{shu2000,foseca2007}.
    Homogenization of the elliptic system \eqref{eq1} was first studied  by Bensoussan, Lions, and Papanicolaou in \cite{lions1978}, 
   where qualitative results were  obtained for the case $\kappa=\va$. Later on, in
\cite{francfort1994} Francfort and  M\"{u}ller provided a systematic qualitative analysis in periodic homogenization of \eqref{eq1} 
and the related nonlinear functionals for the case $\kappa=\va^\gamma$, where $ 0<\gamma<\infty$. See also \cite{zeppieri2016stochastic} for the related work in the stochastic setting. 
Assume that $A$ satisfies conditions \eqref{econ} - \eqref{pcon} and $\kappa=\va^\gamma$.
Let $u_\va\in H^2_0(\Omega; \R^d)$ be the weak solution of \eqref{eq1} with $F\in H^{-1}(\Omega; \R^d)$.
Thanks to \cite{lions1978,francfort1994}, as  $\va\rightarrow 0$, $u_\va$ converges weakly in $H^1(\Omega; \R^d) $
 to the weak solution $u_0$ in $H^1_0(\Omega)$  of the second-order elliptic system,
 \begin{align}   \label{heq0}
  -\text{div} (\widehat{A} \na u_0 )=F ~\text{ in } \Omega,
 \end{align}
 with constant coefficients.
The effective  coefficient  matrix $\widehat{A}$ in (\ref{heq0})  depends on $\kappa$, which represents the strength of the singular perturbation,
 in three cases: $0<\gamma<1$;  $\gamma=1$;  and $\gamma>1$.
In the case $\gamma>1$, the matrix $\widehat{A}$ agrees with the effective  matrix for the second-order 
elliptic operator $-\text{\rm div} (A(x/\va)\nabla )$, without singular perturbation.
If  $0< \gamma< 1$, the matrix $\widehat{A}$ is simply given by the average of $A$ over its periodic cell.
In the  most interesting case $\gamma=1$, the expression  for the matrix $\widehat{A}$  depends on  a corrector, which solves a cell problem for a fourth-order elliptic 
system. The same is true for a  general $\kappa=\kappa (\va)$ under the assumption that 
\begin{align}\label{ratio}
  \kappa\rightarrow 0  \,\text{ as }\, \va\rightarrow 0, \quad\text{ and } \quad\lim_{\va\rightarrow 0} \frac{\kappa}{\va}=\rho.
\end{align}
The  effective  matrix  $\widehat{A}$ in \eqref{heq0} depends on $\rho $ in three cases: 
$\rho =0$;  $ 0<\rho <\infty$;  and $\rho =\infty.$
See Section \ref{q-homo} for the details.

Our primary interest in this paper is in the quantitative  homogenization of the elliptic system (\ref{eq1}).
The qualitative results described above show that the singular perturbation and the homogenization have combined effects 
in determining the effective equation  for \eqref{eq1}. 
So a natural question is  to understand the combined effects in a quantitative way.
More precisely, we shall be interested in the sharp convergence rate of $u_\va$ to $u_0$ in terms of $\va$ and $\kappa$, as well as regularity estimates of $u_\va$,
which  are uniform in $\va$ and $\kappa$.
Although much work has been done on the quantitative homogenization for the second-order elliptic system $-\text{\rm div}(A(x/\va)\nabla u_\va)=F$
 in recent years,  
to the  best of our  knowledge, the question  has not been previously addressed, with the exception of \cite{niuyuan2019}, 
where an $O(\va)$ rate in $L^2(\Omega)$ was obtained in the case $\kappa=\va$ for Dirichlet problems with homogeneous 
boundary conditions.

Our first main result  provides a convergence rate in $L^2(\Omega)$ for a general $\kappa$ satisfying (\ref{ratio}).

\begin{theorem}\label{coth1}
Let $\Omega$ be a bounded $C^{1,1}$ domain in $\R^d, d\geq2$, 
and $A$ satisfy \eqref{econ}-\eqref{pcon}. 
 Suppose \eqref{ratio} holds and if $\rho =0$, we also assume that $A$ is Lipschitz continuous, i.e.,
\begin{align}\label{hcon}
|A(x)-A(y)|\leq L |x-y| ,  \text{ for any } x,y\in \R^d.
\end{align}
For $F\in L^2(\Omega; \R^d)$ and $G\in H^2(\Omega; \R^d)$, let
 $u_{\varepsilon}\in H^2(\Omega; \R^d) $ be  a weak solution  of \eqref{eq1} 
with $u_\va -G \in H_0^2(\Omega; \R^d)$,  and  $u_0\in H^1(\Omega; \R^d)$ the weak solution of its homogenized problem
\eqref{heq0}  with $u_0 -G \in H_0^1(\Omega; \R^d)$.
Then \begin{equation}\label{re1}
 \|u_{\varepsilon}-u_0\|_{L^{2}(\Omega)}\leq 
 \big\{ \|F\|_{L^2(\Omega)}  + \|G\|_{H^2(\Omega)} \big\} 
 \left\{
 \aligned
 &  C_1  \left\{ \kappa + \va +\left( \frac{\va}{\kappa} \right)^2\right\}     &\text{ if }  \rho =\infty,\\
 & C_2 \left\{ \kappa +\va + \rho^{-2} \big|  \big( \frac{\kappa}{\va} \big)^2  -\rho^2  \big| \right\}  &\text{ if }  0< \rho < \infty,\\
 & C_3 \left\{ \kappa +\va + \left( \frac{\kappa}{\va} \right)^2 \right\}    &\text{ if }  \rho=0,
 \endaligned
 \right.
\end{equation}
where $C_1, C_2$ depend only on $d$, $\nu_1$, $\nu_2$ and  $\Omega$,
and  $C_3$ depends only on $d$,  $\nu_1$, $\nu_2$,  $\Omega$, and $L$.
\end{theorem}

The $O(\va)$ convergence rate in $L^2(\Omega)$  has been established  for second-order elliptic
systems with highly oscillating coefficients  in various contexts.
Following a general  approach developed  in 
 \cite{suslinaD2013, shenzhu2017} (see \cite{shennote2018} for references on the related work),
 one first establishes an $O(\va^{1/2})$ rate in $H^1(\Omega)$ for a two-scale expansion of $u_\va$, and
 then uses a duality argument to improve the rate to $O(\va)$ in $L^2(\Omega)$.
 To carry this out, we introduce  an operator,
\begin{align}\label{lvalm}
\mathcal{L}_\va^\lm=\lm^2 \va^2 \Delta^2 -\text{div} (A(x/\va)  \na   ),
\end{align}
where $0<\lm<\infty$ is fixed. Let   $\mathcal{L}_0^\lm = -\text{\rm div} \big( \widehat{A^\lm}\nabla )$
denote the effective operator   for  $\mathcal{L}_\va^\lm $ in (\ref{lvalm}).
In Section \ref{section-4}  we  will show that  if $\mathcal{L}_\va^\lm (u_{\va, \lm } )=F$ and $u_{\va, \lm}-G \in H^2_0(\Omega; \R^d)$, then 
\begin{equation}\label{conv-01}
\| u_{\va, \lm}  - u_{0, \lm} \|_{L^2(\Omega)}
\le C ( 1+\lm) \va  \big\{ \| F\|_{L^2(\Omega)} + \| G\|_{H^2(\Omega)} \big\},
\end{equation}
where $u_{0, \lm}$ is the weak solution of $\mathcal{L}_0^\lm (u_{0, \lm})=F$ in $\Omega$
with $u_{0, \lm}-G \in H^1_0(\Omega; \R^d)$.
To complete the proof of Theorem \ref{coth1}, we observe that 
\begin{equation}\label{l-lm}
\mathcal{L}_\va^\lm =\mathcal{L}_\va \ \  \text{ and }\ \  u_{\va, \lm} =u_\va \quad \text{ if }  \lm=\kappa \va^{-1},
\end{equation}
and use energy estimates to bound $\|u_{0, \lm} -u_0\|_{L^2(\Omega)}$.

We note that the convergence rate in (\ref{re1}) involves three terms.
The first term $\kappa$ is caused by the singular perturbation, the second term $\va$ by homogenization,
while the third term is generated  by  $|\widehat{A^\lm}- \widehat{A}|$.
One may find examples in the one-dimensional case, which show that both the perturbation error $O(\kappa)$ and
the homogenization error $O(\va)$ are sharp. 
Our estimates of $|\widehat{A^\lm}- \widehat{A}|$ in Section \ref{q-homo} should also be sharp as $\lm \to 0$ or $\infty$.
As a result, we believe the convergence rates obtained   in Theorem \ref{coth1} are sharp.
On the other hand, in view of (\ref{conv-01}),
 it is  interesting to point out that  for any $\va>0$ and $ \kappa>0$,  the solution $u_\va$ may be approximated with an $O(\kappa +\va)$ error in 
  $L^2(\Omega)$ by  the solution of a second-order elasticity system with constant coefficients satisfying (\ref{econ}).
  However, the coefficients depend on  $\lm=\kappa\va^{-1}$.
  
  Our second main result gives the large-scale Lipschitz estimate down to the microscopic scale $\va$.
  
  \begin{theorem}\label{lipth}
Assume that $A$ satisfies \eqref{econ} and \eqref{pcon}.
Let $u_\varepsilon\in H^2(B_R; \R^d  )$ be a  weak solution of 
 $\mathcal{L}_\varepsilon(  u_\varepsilon) =F$ in $B_R$, where  $B_R=B(x_0, R)$, $R>\va$,  and
 $F\in L^p(B_R; \R^d  )$ for some $p>d$. Then for $\va \leq r< R$, 
\begin{align}\label{L-L-0}
 \left(\fint_{B_r} |\na u_\varepsilon |^2\right )^{1/2}\leq C  \left\{ \left (\fint_{B_R}
| \na u_\varepsilon|^2\right)^{1/2} +
R\left (\fint_{B_R} |F|^p\right)^{1/p}\right\},
\end{align}
where $C$ depends only on $d$, $\nu_1$, $\nu_2$, and $p$.
\end{theorem}

 Under the additional smoothness  condition that $A$ is H\"older continuous:
 \begin{equation} \label{H-cond}
 |A(x)-A(y)| \le M |x-y|^\sigma  \quad  \text{ for any } x, y\in \R^d,
 \end{equation}
 we obtain the classical Lipschitz estimate, which is uniform in both $\va$ and $\kappa$,  for $\mathcal{L}_\va (u_\va)=F$.
 
 \begin{theorem}\label{main-thm-3}
 Assume that $A$ satisfies conditions \eqref{econ}, \eqref{pcon}, and \eqref{H-cond} for some $\sigma \in (0 ,1)$.
 Let $u_\va \in H^2(B_r; \R^d)$ be a weak solution of $\mathcal{L}_\va (u_\va)=F$ in $B_r =B(x_0, r)$,
 where $F\in L^p(B_r; \R^d)$ for some $p>d$. Then 
 \begin{equation}\label{T-lip}
 |\nabla u_\va (x_0)|\le C \left\{ \left(\fint_{B_r} |\nabla u_\va|^2 \right)^{1/2}
 + r \left(\fint_{B_r} |F|^p \right)^{1/p} \right\},
 \end{equation}
 where $C$ depends only on $d$, $\nu_1$, $\nu_2$, $p$, and $(M, \sigma)$.
 \end{theorem}
 
 Under the conditions (\ref{econ}), (\ref{pcon}) and (\ref{H-cond}),
 the interior Lipschitz estimate (\ref{T-lip}) as well as the boundary Lipschitz estimate with the Dirichlet condition was   proved by Avellaneda and Lin 
 in a seminal work \cite{al87}, using  a compactness method.
 The boundary Lipschitz estimate with Neumann conditions was established in \cite{klsa1}. 
Related work in the stochastic setting may be found in \cite{Gloria2015, armstrongan2016, armstrongar2016, fisher2016, gloriajems2017}.

To prove Theorem \ref{lipth}, we use an approach found in \cite{fisher2016}.
As in \cite{al87}, the idea is to utilize  correctors to establish a large-scale $C^{1, \alpha}$ estimate for  $0<\alpha<1$, from which the large-scale
Lipschitz estimate (\ref{L-L-0}) follows.
Unlike the compactness method used in \cite{al87,klsa1},
the approach requires a (suboptimal) convergence rate in $H^1(\Omega)$ for a two-scale expansion of $u_\va$.
In order to reach down to the microscopic scale $\va$, which is necessary for obtaining  the classical  Lipschitz estimate
in Theorem \ref{main-thm-3}, we introduce an intermediate equation,
\begin{equation}\label{I-E}
\lm^ 2 \va^2 \Delta^2 v_{\va, \lm}  -\text{\rm div} (\widehat{A^\lm}\nabla v_{\va, \lm} ) =F,
\end{equation}
with $\lm>0$ fixed, where $\widehat{A^\lm}$ is the effective matrix for $\mathcal{L}_\va^\lm$ in  (\ref{lvalm}).
The key observation is to use the solution of (\ref{I-E}), instead of the homogenized equation (\ref{heq0}),
 in the two-scale expansion of $u_\va$.
The purpose is two-fold. Firstly, with the added higher-order term in the equation \eqref{I-E},
one eliminates the  error caused by the singular perturbation.
As a result, we are able  to establish a convergence rate in $H^1(\Omega)$, uniformly in $\lm$.
Secondly,  since $\widehat{A^\lm}$ is constant, one may prove the $C^{1, \alpha}$ estimate, uniformly in $\lm$,
for (\ref{I-E}) by classical methods. We remark that as in \cite{fisher2016}, the same approach may be used to establish the large-scale 
$C^{k, \alpha}$ estimates down to the scale $\va$ for any $k\ge 2$.

The paper is organized as follows.
In Section \ref{section-2} we collect some regularity estimates, which are uniform in $\lm$, for the operator (\ref{L-lm}) without the periodicity 
assumption. The materials in this section are more or less known. 
In Section \ref{q-homo} we present the qualitative homogenization for the operator (\ref{lva}) under the assumption (\ref{ratio}).
The proof of Theorem \ref{coth1} is given in Section \ref{section-4}.
In Section \ref{section-app} we establish an approximation result in $H^1(\Omega)$ for $u_{\va, \lm}$ by solutions of (\ref{I-E}),
while the result is used in Section \ref{section-5} to  prove the large-scale $C^{1, \alpha}$ estimate.
Finally, the proofs of Theorems \ref{lipth} and \ref{main-thm-3} are given in Section \ref{section-6}.

The summation convention is used throughout. We  also use $\fint_E u$ to denote the $L^1$ average of $u$ over the set $E$.


\section{Preliminaries}\label{section-2}

Consider the operator,
\begin{equation}\label{L-lm}
\mathcal{L}^\lm =\mathcal{L}_1^\lm
=\lm^2 \Delta^2 -\text{\rm div} (A(x) \nabla ), 
\end{equation}
with $0< \lm<\infty$ fixed  and $A=A(x)$ satisfying the elasticity condition (\ref{econ}).
The periodicity condition (\ref{pcon}) is not used in this section with the exception of  Lemma \ref{s-pp-lemma} and Theorem \ref{s-pp-theorem}.
Let $\Omega$ be a bounded Lipschitz domain in $\mathbb{R}^d$.
For $F\in H^{-1}(\Omega; \R^d)$ and $G\in H^2(\Omega; \R^d)$, there exists a unique $u\in H^2(\Omega; \R^d)$ such that
$\mathcal{L}^\lm (u)=F$ in $\Omega$ and $u -G\in H^2_0(\Omega; \R^d)$.
Moreover, the solution $u$ satisfies the energy estimate,
\begin{equation}\label{energy-0}
\lm \|\nabla^2 u\|_{L^2(\Omega)}
+ \| \nabla u \|_{L^2(\Omega)}
\le C \big\{ \| F\|_{H^{-1}(\Omega)}
+ \| \nabla G\|_{L^2(\Omega)} +\lm \|\nabla^2 G\|_{L^2(\Omega)}  \big\},
\end{equation}
where $C$ depends only on $d$, $\nu_1$, $\nu_2$, and $\Omega$.
To see this, one considers $v=u -G$ and applies the Lax-Milgram Theorem  to the bilinear form,
\begin{equation}\label{bi}
a(\phi, \psi)
= \lm^2 \int_\Omega \nabla^2 \phi \cdot \nabla^2 \psi\, dx
+ \int_\Omega A(x) \nabla \phi \cdot \nabla \psi \, dx,
\end{equation}
on the Hilbert space $H^2_0(\Omega; \R^d)$.
The first Korn inequality is needed for proving (\ref{energy-0}).

\subsection{Caccioppoli's inequalities}

\begin{theorem}\label{Ca-thm}
Let $u\in H^2(B_{2r}; \R^d) $ be a weak solution of
$\mathcal{L}^\lm (u)= F +\text{\rm div} (f)$ in $B_{2r}=B(x_0, 2r)$,
where $F\in L^2(B_{2r}; \R^d)$ and $ f \in L^2(B_{2r}; \R^{d\times d} )$.
Then
\begin{align}
\lambda^2 \int_{B_r}
|\nabla^2  u|^2\, dx 
& \le  \frac{C}{r^2} \left( \frac{\lambda^2}{r^2} +1 \right) \int_{B_{2r}} |u|^2\, dx + C  \int_{B_{2r}} |F| |u| \, dx
+ C \int_{B_{2r}} |f|^2\, dx , \label{Ca-1}\\
\int_{B_r} |\nabla u|^2\, dx
& \le \frac{C}{r^2} \int_{B_{2r}} |u|^2\, dx
+ C   \int_{B_{2r}} |F| |u| \, dx
+ C \int_{B_{2r}} |f|^2\, dx \label{Ca-2},
\end{align}
where $C$ depends only on $d$, $\nu_1$ and $\nu_2$.
\end{theorem}

\begin{proof}

By translation and dilation we may assume that $x_0=0$ and $r=1$.
For $1<s<t< 2$, let $\varphi$ be a cut-off  function in $C_0^\infty(B(0, t))$ such that $0\le \varphi \le 1$, 
$\varphi =1$ on $B_s$ and $|\nabla^k \varphi| \le C (t-s)^{-k}$ for $k=1, \dots , 4$.
By taking the test function $ u \varphi^4$ in the weak formulation of the  equation
$\mathcal{L}^\lm (u)= F +\text{\rm div}(f)$ and using the Cauchy inequality, we deduce that
\begin{equation}\label{Ca-3}
\aligned
& \lm^2 \int_{B_s} |\nabla^2 u|^2\, dx
+\int_{B_s} |\nabla u|^2\, dx\\
&\le
C \int_{B_2}  \big(   |F| |u| +| f|^2  \big)  \, dx
+C \lm^2 (t-s)^{-2} \int_{B_t} |\nabla (u\varphi) |^2\, dx \\
& \qquad
+C\big( (t-s)^{-2} + 
 \lm^2 (t-s)^{-4} \big) \int_{B_2} |u|^2\, dx.
\endaligned
\end{equation}
To eliminate the term involving $|\nabla (u\varphi) |$ in the right-hand side of (\ref{Ca-3}), 
we use an iteration technique found in \cite{barton2016}, where  an improved Caccioppoli inequality for a general higher-order
elliptic system was proved.
We point out that Theorem \ref{Ca-thm} does not follow directly from \cite{barton2016}, since we require the constant $C$ to be  independent of the 
parameter $\lm$.

Using the identity, 
$$
u \varphi \cdot \Delta (u\varphi)
=(u\Delta u) \varphi^2
+ 2 u \nabla (u\varphi) \nabla  \varphi
- 2 |u|^2 |\nabla \varphi|^2
+ |u|^2 \varphi \Delta \varphi,
$$
and integration by parts as well as the Cauchy inequality, we may show that
\begin{equation}\label{Inter}
\aligned
\int_{B_t} |\nabla (u\varphi)|^2\, dx
 & \le  C \left(\int_{B_t} |  u\varphi   |^2\, dx \right)^{1/2}
\left(\int_{B_t} |\varphi \Delta u|^2 \, dx \right)^{1/2}\\
& \qquad
+ C \int_{B_t}
| u|^2 |\nabla \varphi |^2 \ dx 
 +C \int_{B_t}
 |u|^2 |\varphi| |\Delta \varphi|\, dx,
 \endaligned
 \end{equation}
 where $C$ depends only on $d$.
This, together with (\ref{Ca-3}), gives
\begin{equation}\label{Ca-4}
\aligned
& \lm^2 \int_{B_s} |\nabla^2 u|^2\, dx
+\int_{B_s} |\nabla u|^2\, dx\\
&
\le
C \int_{B_2}  \big(   |F| |u|  +| f|^2  \big)  \, dx
+ \frac{\lm^2}{2} 
\int_{B_t} |\nabla^2 u|^2\, dx  \\
& \qquad
+C\big( (t-s)^{-2} 
+ \lm^2 (t-s)^{-4} \big) \int_{B_2} |u|^2\, dx.
\endaligned
\end{equation}
For $j\ge 1$, let 
$
t_j = 2-\tau^j  ,
$
where $\tau \in (0, 1)$ is to be determined.
It follows from (\ref{Ca-4}) that
\begin{equation}\label{Ca-5}
\aligned
& \lm^2 \int_{B_{t_j} } |\nabla^2 u|^2\, dx
+\int_{B_{t_j} } |\nabla u|^2\, dx\\
&
\le
C \int_{B_2}  \big(   |F| |u|  +| f|^2  \big)  \, dx
+ \frac{\lm^2}{2} 
\int_{B_{t_{j+1} }} |\nabla^2 u|^2\, dx  \\
& \qquad
+C\big( (\tau^j -\tau^{j+1} )^{-2} 
+ \lm^2 (\tau^j -\tau^{j+1} )^{-4} \big) \int_{B_2} |u|^2\, dx.
\endaligned
\end{equation}
By iteration this leads to
\begin{equation}\label{Ca-6}
\aligned
& \lm^2 \int_{B_{t_1} } |\nabla^2 u|^2\, dx
+\int_{B_{t_1} } |\nabla u|^2\, dx\\
&
\le
C\sum_{i=1}^j \frac{1}{2^{i-1}}
 \int_{B_2}  \big(   |F| |u|  +| f|^2  \big)  \, dx
+ \frac{\lm^2}{2^j} 
\int_{B_{t_{j+1} }} |\nabla^2 u|^2\, dx  \\
& \qquad
+C\sum_{i=1}^j
\frac{1}{2^{i-1}}
\big( (\tau^i -\tau^{i+1} )^{-2} 
+ \lm^2 (\tau^i -\tau^{i+1} )^{-4} \big) \int_{B_2} |u|^2\, dx
\endaligned
\end{equation}
for $j\ge 1$.
We now choose $\tau\in (0, 1)$ so that $2\tau^4>1$.
By letting $j \to \infty$ in (\ref{Ca-6}) we obtain (\ref{Ca-1}) with $r=1$, and
\begin{equation}\label{Ca-7}
\int_{B_1} |\nabla u|^2\,dx
\le C (\lm^2 +1) \int_{B_2} |u|^2\, dx
+ C \int_{B_2}
\big( |F| |u|  + |f|^2\big)\, dx,
\end{equation}
which gives (\ref{Ca-2}) if $\lm\le 1$.
Finally, if $\lm>1$, we note that (\ref{Ca-6}) yields
\begin{equation}\label{Ca-8}
\lm^2 \int_{B_{t_1}}
|\nabla^2 u|^2\, dx
\le C \int_{B_2} 
\big( |F| |u|  +|f|^2  \big)\, dx
+ C (1+\lm^2) \int_{B_2} |u|^2\, dx.
\end{equation}
By (\ref{Inter}) we have
\begin{equation}\label{Ca-9}
\aligned
\int_{B_1}
|\nabla u|^2\,  dx 
&\le C \int_{B_{t_1}} |u|^2\, dx + C \int_{B_{t_1}} |\Delta  u|^2\, dx\\
&\le C \int_{B_2}
\big( |F| |u| +|f|^2 + |u|^2 \big)\, dx,
\endaligned
\end{equation}
where we have used (\ref{Ca-8}) for the last inequality.
\end{proof}

\begin{remark}\label{remark-higher}

{\rm 

Let $u$ be a solution of $\mathcal{L}^\lm (u)  =F +\text{\rm div}(f)$ in $B_{2r}$.
Let $w=\lambda^2 \Delta u$.
Since
$$
\Delta w = F +\text{\rm div} (f) +\text{\rm div} (A\nabla u),
$$
it follows from the Caccioppoli inequality for $\Delta$  that
\begin{equation}\label{high-e}
\aligned
\int_{B_r}
\lm^4  |\nabla \Delta u |^2\, dx
 & \le \frac{C\lm^4}{r^2}
\int_{B_{3r/2}} |\Delta u|^2\, dx
+ C r^2 \int_{B_{3r/2}}|F|^2\, dx\\
&\qquad \qquad
+C \int_{B_{3r/2}} |f|^2\, dx
+ C \int_{B_{3r/2}} |\nabla u|^2\, dx\\
&\le
\frac{C}{r^2}
\left( \frac{\lm}{r} + 1\right)^4 \int_{B_{2r}} |u|^2\, dx
+C \left(\frac{\lm}{r} + 1 \right)^2 \int_{B_{2r}} |f|^2\, dx\\
& \qquad\qquad
+ C r^2\int_{B_{2r}} |F|^2\, dx,
\endaligned
\end{equation}
where we have used (\ref{Ca-1}) and (\ref{Ca-2}) for the last inequality.
}
\end{remark}

\subsection{Reverse H\"older  inequalities}

\begin{theorem}\label{RH-thm}
Let $u\in H^2(B_{2r}; \R^d) $ be a weak solution of
$\mathcal{L}^\lm (u)=F +\text{\rm div}(f)$ in $B_{2r} =B(x_0, 2r)$,
where $F\in L^2(B_{2r}; \R^d)$ and  $ f\in L^2(B_{2r}; \R^{d\times d} )$.
Then there exists some $p>2$, depending only on $d$, $\nu_1$ and $\nu_2$, such that
\begin{equation}\label{RH-1}
\left(\fint_{B_r}
|\nabla u|^p \right)^{1/p}
\le C \left\{
\left(\fint_{B_{2r}} |\nabla u|^2 \right)^{1/2}
+ \left(\fint_{B_{2r}} |f|^p\right)^{1/p}
+C r \left(\fint_{B_{2r}}
|F|^2 \right)^{1/2} \right\},
\end{equation}
where $C$ depends only on $d$, $\nu_1$ and $\nu_2$.
\end{theorem}

\begin{proof}
This follows from (\ref{Ca-2}) by the  self-improvement property of the (weak) reverse H\"older 
inequalities.  Let $B^\prime =B(z, t)$ be a ball such that $2B^\prime \subset B(x_0, 2r)$.
Choose $1<q_1<2<q_2<\infty$ such that
$$
\left(\fint_{2B^\prime}
|u-E|^{q_2}\right)^{1/q_2}
\le C t  \left(\fint_{2B^\prime}
|\nabla u|^{q_1} \right)^{1/q_1},
$$
where $E$ is the $L^1$ average of $u$ over $2B^\prime$.
Since $\mathcal{L}^\lambda (u-E)=\mathcal{L}^\lm (u)$, it follows from (\ref{Ca-2}) that
\begin{equation}\label{RH-2}
\left(\fint_{B^\prime}
|\nabla u|^2\right)^{1/2}
\le C \left(\fint_{2B^\prime}
|\nabla u|^{q_1} \right)^{1/q_1}
+ Ct \left(\fint_{2B^\prime}
|F|^{q_2^\prime} \right)^{1/q_2^\prime}
+C \left(\fint_{2B^\prime}
|f|^2 \right)^{1/2},
\end{equation} 
where $C$ depends only on $d$, $\nu_1$ and $\nu_2$.
The fact that (\ref{RH-2}) holds for any ball $2B^\prime\subset B$ implies (\ref{RH-1})  \cite{Gia-1983}.
\end{proof}

\begin{remark}\label{b-Ca}

{\rm

Let $\Omega$ be a bounded Lipschitz domain.
Fix $x_0\in \partial\Omega$ and define 
$$
D_r =B(x_0, r) \cap \Omega \quad \text{ and } \quad \Delta_r= B(x_0, r) \cap \partial\Omega,
$$
where $0<r< r_0= c_0 \, \text{\rm diam}(\Omega)$.
Let $u\in H^2(D_{2r}; \R^d)$ be a weak solution of
$\mathcal{L}^\lm (u)= F +\text{\rm div} (f)$ in $D_{2r}$ with 
$u=0$ and $\nabla u=0$ on $\Delta_{2r}$. Then 
\begin{align}
\lambda^2 \int_{D_r}
|\nabla^2  u|^2\, dx 
& \le  \frac{C}{r^2} \left( \frac{\lambda^2}{r^2} +1 \right) \int_{D_{2r}} |u|^2\, dx + C  \int_{D_{2r}} |F| |u| \, dx
+ C \int_{D_{2r}} |f|^2\, dx , \label{Ca-20}\\
\int_{D_r} |\nabla u|^2\, dx
& \le \frac{C}{r^2} \int_{D_{2r}} |u|^2\, dx
+ C   \int_{D_{2r}} |F| |u| \, dx
+ C \int_{D_{2r}} |f|^2\, dx \label{Ca-21},
\end{align}
where $C$ depends only on $d$, $\nu_1$ and $\nu_2$.
Note that since $u=0$ and $\nabla u=0$ on $\Delta_{2r}$, 
we have $u\varphi \in H_0^2(D_{2r}; \R^d)$ for any $\varphi\in  C_0^2(  B_{2r}) $.
The proof of (\ref{Ca-20}) and (\ref{Ca-21})  is exactly the same as that of Theorem \ref{Ca-thm}.
As a consequence, we also obtain the boundary reverse H\"older inequality, 
\begin{equation}\label{RH-b}
\left(\fint_{D_r}
|\nabla u|^p \right)^{1/p}
\le C \left\{
\left(\fint_{D_{2r}} |\nabla u|^2 \right)^{1/2}
+ \left(\fint_{D_{2r}} |f|^p\right)^{1/p}
+C r \left(\fint_{D_{2r}}
|F|^2 \right)^{1/2} \right\},
\end{equation}
where $C>0$ and $p>2$ depend only on $d$, $\nu_1$, $\nu_2$ and the Lipschitz constant of $ B(z, r_0)\cap \partial \Omega $.
}
\end{remark}

\begin{theorem}\label{Meyers-thm}
Suppose $A$ satisfies \eqref{econ} and $\Omega$ is a bounded Lipschitz domain.
Let $u\in H^2_0(\Omega; \R^d)$ be a weak solution of $\mathcal{L}^\lm (u)=\text{\rm div}(f)$ in 
$\Omega$. Then there exists $p>2$, depending only on $d$, $\nu_1$, $\nu_2$ and $\Omega$, 
such that
\begin{equation}\label{Meyers-e}
\|\nabla u\|_{L^p(\Omega)}
\le C \| f\|_{L^p(\Omega)},
\end{equation}
where $C$ depends only on $d$, $\nu_1$, $\nu_2$, and $\Omega$.
\end{theorem}

\begin{proof}
The Meyers estimate (\ref{Meyers-e}) was proved in \cite{francfort1994} by an interpolation argument.
It also follows readily from the reverse H\"older estimates (\ref{RH-1}) and (\ref{RH-b}).
Indeed, by using (\ref{RH-1}), (\ref{RH-b}) and a simple covering argument, we see that for some $p>2$,
$$
\|\nabla u\|_{L^p(\Omega)}
\le C \| f\|_{L^p(\Omega)} + C \|\nabla u\|_{L^2(\Omega)}
\le C \| f\|_{L^p(\Omega)},
$$
where we have used the energy estimate and H\"older's  inequality for the last step.
\end{proof}

\subsection{$C^{1, \alpha}$ estimates}

\begin{lemma}\label{lemma-r-1}
Suppose $A$ satisfies conditions \eqref{econ}  and \eqref{H-cond}.
Let $u\in H^2(B_2; \R^d)$ be a weak solution of $\mathcal{L}^\lm (u)=0$ in $B_2=B(0, 2)$.
Then
\begin{equation}\label{C-1-a}
\| u\|_{C^{1, \alpha} (B_1)}
\le C_\alpha  \left(\fint_{B_2}  |u|^2\right)^{1/2},
\end{equation}
where $0<\alpha< \sigma$ and $C_\alpha $ depends only on $d$, $\nu_1$, $\nu_2$, $\alpha$, and $(M, \sigma)$.
\end{lemma}

\begin{proof}
We first observe that if  $A$ is a constant matrix satisfying the elasticity condition (\ref{econ}), then
\begin{equation}\label{C-1a-1}
\max_{B_1} |\nabla^k u|
\le C_k\left(\fint_{B_{{3/2}}} |u|^2\right)^{1/2},
\end{equation}
where $C_k$ depends on $d$, $\nu_1$, $\nu_2$ and $k$.
To see this, we note that since $A$ is constant,  $\nabla^k u$ is a solution.
Thus, by (\ref{Ca-2}) and an iteration argument,
$$
\| u\|_{H^k (B_1)}
\le C_k \| u\|_{L^2(B_{3/2})}
$$
for any $k\ge 1$.  By Sobolev imbedding, this gives (\ref{C-1a-1}).
Next, we use a standard perturbation argument to show that  if $A$ is uniformly continuous and $\gamma>0$,
\begin{equation}\label{C-1a-2}
\int_{B_\rho} |\nabla u|^2\, dx
\le C_\gamma \left(\frac{\rho}{R} \right)^{d-2\gamma}
\int_{B_R} |\nabla u|^2\, dx
\end{equation}
for $0< \rho< R< r$. To do this, we let $v\in H^2(B_R; \R^d)$ be the solution of
\begin{equation}\label{C-1a-3}
\lm^2 \Delta^2 v -\text{\rm div} (\overline{A} \nabla v) =0
\quad
\text{ in } B_R \quad 
\text{ and } v-u \in H^2_0(B_R ; \R^d),
\end{equation}
where 
$
\overline{A}=\fint_{B_R} A.
$
Since
$$
\lm^2 \Delta^2  (v-u) -\text{\rm div} (\overline{A} \nabla (v-u )) =\text{\rm div}( (\overline{A} -A)\nabla u) \quad \text{ in } B_R,
$$
by energy estimates,
$$
\int_{B_R} 
|\nabla u -\nabla v|^2\, dx
\le C \| \overline{A} -A\|^2_{L^\infty(B_R)} \int_{B_R} |\nabla u|^2\, dx .
$$
By (\ref{C-1a-1}), for $0<\rho< R< r$, 
$$
\fint_{B_\rho } |\nabla v|^2
\le C \fint_{B_R} |\nabla v|^2.
$$
The rest of the argument for (\ref{C-1a-2}) is exactly the same as in the case of second-order elliptic systems \cite[pp.84-88]{Gia-1983}.
An  argument  similar to that in \cite[pp.84-88]{Gia-1983} also shows that if $A$ satsifies (\ref{H-cond}), then
$$
\left(\fint_{B_r} |\nabla u
-\fint_{B_r} \nabla u  |^2 \right)^{1/2}
\le C_\alpha  r^\alpha \left(\fint_{B_2} |u|^2 \right)^{1/2}
$$ 
for any $\alpha \in (0, \sigma)$ and $0<r<1$.
This implies (\ref{C-1-a}).
\end{proof}
 
 The following theorem gives the $C^{1, \alpha}$ estimate, uniform in $\lm$, for the operator $\mathcal{L}^\lm$.
 
\begin{theorem}\label{loc-thm}
Suppose $A$ satisfies conditions \eqref{econ}  and \eqref{H-cond}.
Let $u \in H^2(B_2; \R^d )$ be a weak solution of
$\mathcal{L}^\lm (u) = F$ in $B_2$, where $F\in L^p(B_2; \R^d)$ for some $p>d$.
Then, if $0< \alpha <\min (\sigma , 1-\frac{d}{p})$, 
\begin{equation}\label{loc-est}
\| u\|_{C^{1, \alpha}(B_1)}
\le C_\alpha  \Big\{ \| u\|_{L^2(B_2)}
+ \|F\|_{L^p(B_2)} \Big\}, 
\end{equation}
where $C_\alpha$ depends on $d$, $\nu_1$, $\nu_2$, $p$, $\alpha$, and $(M,\sigma)$.
\end{theorem}

\begin{proof}

The case $F=0$ was given by Lemma \ref{lemma-r-1}.
The general case is proved by a perturbation argument as in the case of second-order elliptic systems.
Let $0<r<R<1$.
Let $v\in H^2(B_R;\R^d)$ be the weak solution of
$\mathcal{L}^\lm (v)=0$ in $B_R$ such that $v-u\in H^2_0(B_R; \R^d)$.
Since $\mathcal{L}^\lm (u-v)=F$ in $B_R$, by the energy estimate,
\begin{equation}\label{local-10}
\int_{B_R} |\nabla u -\nabla v|^2\, dx 
\le C R^2  \int_{B_R} |F|^2\, dx
\le C R^{d+ 2(1-\frac{d}{p})} \| F\|^{p/2} _{L^p(B_2)},
\end{equation}
where $C$ depends only on $d$, $\nu_1$, $\nu_2$, and $p$.
By Lemma \ref{lemma-r-1},
$$
\int_{B_r}
|\nabla v -\fint_{B_r} \nabla v|^2\, dx
\le C \left(\frac{r}{R}\right)^{d+2 \alpha} 
\int_{B_R}
|\nabla v -\fint_{B_R} \nabla v|^2\, dx
$$
for any $0<\alpha<\sigma$.
This, together with (\ref{local-10}), leads to
$$
\aligned
\int_{B_r}
|\nabla u -\fint_{B_r} \nabla u |^2\, dx
 & \le C \left(\frac{r}{R}\right)^{d+2 \alpha}
\int_{B_R}
|\nabla u -\fint_{B_R} \nabla u|^2\, dx\\
&\qquad
+  C R^{d+ 2(1-\frac{d}{p})} \| F\|^{p/2} _{L^p(B_2)},
\endaligned
$$
from which the estimate (\ref{loc-est}) follows, as in \cite[pp.88-89]{Gia-1983}. We omit the details.
\end{proof}

\subsection{Singular perturbations}

For $\Omega\subset  \R^d$ and $0<t< c_0 \text{diam}(\Omega)$, let
\begin{align} \label{omegava}
\Omega_t=\{x\in\Omega: \text{\rm  dist} (x, \partial \Omega)<t\} .
 \end{align}

\begin{lemma}\label{lsmooth4}
Let $\Omega$ be a bounded Lipschitz  domain in $\R^d$.
Then, 
\begin{align}
\| u\|_{L^2(\Omega_t)}  & \le C t   \|\nabla u\|_{L^2(\Omega_{2t})}  & \text{ for } u\in H^1_0(\Omega), \label{241-0}\\
\| u\|_{L^2(\Omega_t) } & \le C t^{1/2} \| u \|_{L^2(\Omega)}^{1/2}
\| u\|_{H^1(\Omega)}^{1/2} \label{241-1}
 & \text{ for } u\in H^1(\Omega),
\end{align}
and for  $u\in H^2(\Omega)\cap H^1_0(\Omega)$,
\begin{align}\label{241}
\|  u \|_{L^{2}(\Omega_t)}\leq Ct^{3/2}  \|u\|^{1/2}_{H^1(\Omega)} \|u\|^{1/2}_{H^{2}(\Omega)},
\end{align} 
where $C$ depends on $d$ and $\Omega$.
\end{lemma}

\begin{proof}
The inequalities  (\ref{241-0}) and (\ref{241-1})  may be proved  by  a localization argument, while (\ref{241}) follows readily from (\ref{241-0})-(\ref{241-1}).
\end{proof}

\begin{lemma}\label{s-p-lemma}
Let $u_\lm \in H^2(\Omega; \R^d)$ be a weak solution of $\mathcal{L}^\lm (u_\lm)=F$ 
with $u_\lm-G\in H^2_0(\Omega; \R^d)$, where $F\in L^2(\Omega; \R^d )$, $G\in H^2(\Omega; \R^d)$, and 
$\Omega$ is a bounded Lipschitz domain.
Let $u_0\in H^1(\Omega; \R^d)$ be the weak solution of $-\text{\rm div} (A\nabla u_0)=F$ in $\Omega$ 
and $u_0-G \in H^1_0(\Omega; \R^d)$.
Suppose $u_0\in H^2(\Omega; \R^d)$.
Then for $0<\lm \le 1$,
\begin{equation}\label{s-p-1}
\| \nabla u_\lm -\nabla u_0\|_{L^2(\Omega)}
\le C \sqrt{\lm}\big\{ \| u_0 \|_{H^2(\Omega)} + \|G \|_{H^2(\Omega)} \big\},
\end{equation}
where $C$ depends only on $d$, $\nu_1$, $\nu_2$,  and $\Omega$.
\end{lemma}

\begin{proof}
Let $\eta_t$ be a cut-off function in $C_0^\infty(\Omega)$ such that
$0\le \eta_t\le 1$, $\eta_t (x)=1$ if  $x\in  \Omega\setminus \Omega_{2t}$,
$\eta_t(x)=0$ if $x\in \Omega_t$, and
$|\nabla^k \eta_t|\le C t^{-k}$ for $k=1, 2$, where $t>0$ is to be determined.
Let  $\widetilde{u}_0 =u_0-G$ and 
\begin{equation}\label{s-p-2}
 w=u_\lm -G - (u_0-G) \eta_t
 =u_\lm -u_0 + \widetilde{u}_0 (1-\eta_t).
 \end{equation}
Note that $w\in H^2_0(\Omega; \R^d )$ and
$$
\aligned
\mathcal{L}^\lm (w)
&=\mathcal{L}^\lm (u_\lm) -\mathcal{L}^\lm (u_0)   + \mathcal{L}^\lm \big[ \widetilde{u}_0 (1-\eta_t)\big]\\
&=-\lm^2 \Delta^2 u_0
+\lm^2 \Delta^2  ( \widetilde{u}_0 (1-\eta_t ))
-\text{\rm div} \big[ A\nabla (\widetilde{ u} _0 (1-\eta_t)) \big].
\endaligned
$$
It follows that for any $\psi\in H_0^2(\Omega; \R^d)$,
$$
\aligned
 | \langle \mathcal{L}^\lm (w) , \psi \rangle |
 & \le \lm^2 \int_\Omega |\Delta u_0 | |\Delta \psi|\, dx
+ \lm^2 \int_\Omega 
|\Delta ( \widetilde{u}_0 (1-\eta_t))| |\Delta \psi|\, dx\\
& \qquad \qquad
+ C \int_\Omega |\nabla (\widetilde{u} _0 (1-\eta_t))| |\nabla \psi|\, dx.
\endaligned
$$
By using the Cauchy inequality and Lemma \ref{lsmooth4}, we obtain 
\begin{equation}\label{s-p-3}
\aligned
 | \langle \mathcal{L}^\lm (w) , \psi \rangle |
  & \le 
  \lm^2 \|u_0\|_{H^2(\Omega)} \| \Delta \psi\|_{L^2(\Omega)}
 + C \lm^2 t^{-1/2} \| \widetilde{ u } _0\|_{H^2(\Omega)}  \|\Delta \psi \|_{L^2(\Omega_{2t})}\\
 & \qquad\qquad
  + C t^{1/2} \| \widetilde{u}_0\|_{H^2(\Omega)} \|\nabla \psi\|_{L^2(\Omega_{2t})}.
\endaligned
\end{equation}
By taking $\psi=w$ in (\ref{s-p-3}), $t=c_0 \lm$,  and using the Cauchy inequality,  we see that 
\begin{equation}\label{s-p-4}
\lm \| \Delta w\|_{L^2(\Omega)}
+\|\nabla w\|_{L^2(\Omega)}
\le C \lm^{1/2}
\big\{ \| u_0\|_{H^2(\Omega)}
+ \|G \|_{H^2(\Omega)}  \big\}.
\end{equation}
In view of (\ref{s-p-2}) this gives (\ref{s-p-1}).
\end{proof}

\begin{theorem}\label{s-p-theorem}
Let $u_\lm$ and $u_0$ be the same as in Lemma \ref{s-p-lemma}.
Also assume that $\Omega$ is a bounded $C^{1, 1}$ domain and $\|\nabla A\|_\infty \le L<\infty$.
Then for $0< \lm\le 1$,
\begin{equation}\label{s-p-10}
\| u_\lm - u_0 \|_{L^2(\Omega)}
\le C \lm  \big\{ \| F\|_{L^2(\Omega)}
+ \| G\|_{H^2(\Omega)} \big\},
\end{equation}
where $C$ depends on $d$, $\nu_1$, $\nu_2$, $L$, and $\Omega$.
\end{theorem}

\begin{proof}
Let $w$ be given by (\ref{s-p-2}) with $t=c_0 \lm$.
For $\widetilde{F} \in L^2(\Omega; \R^d)$,
let  $\widetilde{w} = v_\lm -v_0  \widetilde{\eta}_t$, where 
$v_\lm\in H_0^2(\Omega;\R^d)$ is the weak solution of $
\mathcal{L}^\lm (v_\lm )=\widetilde{F}$ in $\Omega$ and
$v_0\in H^1_0(\Omega; \R^d)$  the weak solution of $-\text{\rm div}(A\nabla v_0)=\widetilde{F} $ in $\Omega$.
The function $\widetilde{\eta}_t\in C_0^\infty (\Omega)$ is chosen so that $ 0\le \widetilde{\eta}_t \le 1$,
$\widetilde{\eta}_t =1$ in $\Omega\setminus \Omega_{3t}$, $\widetilde{\eta}_t =0$ in $\Omega_{2t}$, 
and $|\nabla^k \widetilde{\eta}_t|\le C t^{-k}$ for $k=1,2$.
Note that
$$
\aligned
\Big|\int_\Omega
w\cdot \widetilde{F} \, dx \Big|
& = | \langle \mathcal{L}^\lm (w), v_\lm \rangle |\\
&\le |\langle \mathcal{L}^\lm (w), \widetilde{w} \rangle |
+|\langle \mathcal{L}^\lm (w), v_0 \widetilde{\eta}_t \rangle |.
\endaligned
$$
It follows from (\ref{s-p-4}) that
$$
\aligned
|\langle \mathcal{L}^\lm (w), \widetilde{w} \rangle |
&\le C \big\{ \lm \|\Delta w\|_{L^2(\Omega)}
+ \|\nabla w\|_{L^2(\Omega) } \big\}
 \big\{ \lm \|\Delta \widetilde{w} \|_{L^2(\Omega)}
+ \|\nabla \widetilde{w} \|_{L^2(\Omega) } \big\}\\
& \le C \lm \big\{ \| u_0\|_{H^2(\Omega)} + \| G\|_{H^2(\Omega)} \big\}  \| v_0\|_{H^2(\Omega)}.
\endaligned
$$
Also, by (\ref{s-p-3}) and the fact that $\widetilde{\eta}_t =0$ in $\Omega_{2t}$,
$$
\aligned
|\langle \mathcal{L}^\lm (w), v_0 \widetilde{\eta}_t \rangle |
&\le  \lm^2 \| u_0\|_{H^2(\Omega)}  \| v_0 \widetilde{\eta}_t \|_{H^2(\Omega)}\\
& \le C \lm^2 t^{-1/2} \| u_0\|_{H^2(\Omega)} \| v_0\|_{H^2(\Omega)},
\endaligned
$$
where we have used Lemma \ref{lsmooth4} for the last inequality.
As a result, we have proved that
$$
\aligned
\Big|\int_\Omega
w\cdot \widetilde{F}  \, dx \Big|
 & \le C \lm \big\{ \| u_0\|_{H^2(\Omega)} + \| G\|_{H^2(\Omega)} \big\}  \| v_0\|_{H^2(\Omega)} \\
 & \le C \lm \big\{ \| u_0\|_{H^2(\Omega)} + \| G\|_{H^2(\Omega)} \big\}  \| \widetilde{F} \|_{L^2(\Omega)},
 \endaligned
 $$
 where, for the last step, we have used the $H^2$ estimate $\|v_0 \|_{H^2(\Omega)} \le C \| \widetilde{F} \|_{L^2(\Omega)}$,
 which holds under the assumption that $A$ is Lipschitz continuous and $\Omega$ is $C^{1, 1}$.
 The estimate (\ref{s-p-10})  now follows readily by duality.
\end{proof}

A proof for Theorem \ref{s-p-theorem}  in the case $d=2$ may be found in \cite{schuss1976}.
As pointed out by A. Friedman in \cite{friedman1968}, the one-dimensional  example,
$$
\left\{
\aligned
 & \lm^2 \frac{d^4u}{dx^4}
-\frac{d^2 u}{dx^2}=1 \quad \text{ in } (0, 1),\\
& u(0)=u(1)=u^\prime (0)=u^\prime (1)=0,
\endaligned
\right.
$$
shows that the $O(\lm)$ rate in  \eqref{s-p-10} is sharp.
However, in the case of periodic boundary conditions, the rates in Lemma \ref{s-p-lemma} and Theorem \ref{s-p-theorem} can be improved.

Let $C^\infty_\per  (\R^d; \R^d)$ denote the space of $C^\infty$, 1-periodic $\R^d$-valued  functions in $\R^d$. Let 
 $H^k_\per (Y; \R^d)$ be   the closure of $C^\infty_\per (\R^d; \R^d)$ in $H^k(Y; \R^d)$, where  $k\ge 1$ and  $Y=[0, 1]^d$. 
 Note that for any $F\in L^2(Y; \R^d)$ with $\int_Y F\, dx =0$, there exists a unique $u_\lm\in H^2_\per (Y; \R^d)$ such that
 $\mathcal{L}^\lm (u_\lm) =F$ in $Y$ and $\int_Y u_\lm \, dx =0$.

\begin{lemma}\label{s-pp-lemma}
Suppose $A$ satisfies conditions \eqref{econ}  and \eqref{pcon}.
Let $u_\lm\in H^2_\per  (Y; \R^d)$ be a weak solution of
$\mathcal{L}^\lm (u_\lm )=F$ in $Y$ with $\int_Y u_\lm\, dx =0$, where $F\in L^2(Y; \R^d)$ and $\int_Y F\, dx=0$.
Let $u_0\in H_\per^1(Y; \R^d)$ be the weak solution of $-\text{\rm div}(A\nabla u_0)=F$ in $Y$ 
with $\int_Y u_0\, dx=0$.
Suppose $u_0\in H_\per ^2(Y; \R^d)$. Then 
\begin{equation}\label{s-pp-1}
\| \nabla u_\lm -\nabla u_0\|_{L^2(Y)}
\le C \lm \| u_0\|_{H^2(Y)},
\end{equation}
where $C$ depends only on $d$, $\nu_1$ and $\nu_2$.
\end{lemma}

\begin{proof}

Let $w=u_\lm -u_0$. Then
$$
\mathcal{L}^\lm (w)=-\lm^2 \Delta^2 u_0.
$$
It follows that for any $\psi \in H^2_\per(Y; \R^d)$,
\begin{equation}\label{s-pp-3}
|\langle \mathcal{L}^\lm (w), \psi \rangle |
\le \lm^2  \| \Delta u_0\|_{L^2(Y)} \|\Delta \psi\|_{L^2(Y)}.
\end{equation}
By taking $\psi =w$ in \eqref{s-pp-3}  and using the Cauchy inequality, we obtain 
\begin{equation}\label{s-pp-4}
\lm \|\Delta w \|_{L^2(Y)}
+ \| \nabla  w \|_{L^2(Y)}
\le C\lm  \| u_0\|_{H^2(Y)},
\end{equation}
which yields (\ref{s-pp-1}).
\end{proof}

\begin{theorem}\label{s-pp-theorem}
Suppose $A$ satisfies \eqref{econ} and \eqref{pcon}.
Also assume that $\|\nabla A\|_\infty\le L < \infty$.
Let $u_\lm$ and $u_0$ be the same as in Lemma \ref{s-pp-lemma}.
Then 
\begin{equation}\label{s-pp-10}
\| u_\lm -u_0\|_{L^2(Y)}
\le C \lm^2 \| F\|_{L^2(Y)},
\end{equation}
where $C$ depends on $d$, $\nu_1$, $\nu_2$, and $L$.
\end{theorem}

\begin{proof}
The proof is similar to that of Theorem \ref{s-p-theorem}.
For $\widetilde{F}\in L^2(Y; \R^d)$ with $\int_Y \widetilde{F}\, dx =0$,
let $\widetilde{w}= v_\lm -v_0$, where $v_\lm\in H^2_\per (Y; \R^d)$ is the weak solution 
of $\mathcal{L}^\lm (v_\lm) =\widetilde{F} $ in $Y$ with $\int_Y v_\lm \, dx =0$, and 
$v_0\in H^1_{\per}(Y; \R^d)$ the solution of $-\text{\rm div}(A\nabla v_0)=\widetilde{F}$ in $Y$ with
$\int_Y v_0\, dx =0$.
Note that
$$
\aligned
\Big| \int_Y
w \cdot \widetilde{F} \, dx \Big|
&= |\langle \mathcal{L}^\lm (w), v_\lm \rangle |\\
&\le  |\langle \mathcal{L}^\lm (w) , \widetilde{w} \rangle |
+ |\langle \mathcal{L}^\lm (w) , v_0 \rangle |.
\endaligned
$$
It follows from (\ref{s-pp-3}) that
$$
\aligned
|\langle \mathcal{L}^\lm (w) , \widetilde{w} \rangle |
&\le \lm^2 \| \Delta w \|_{L^2(Y)} \|\Delta \widetilde{w} \|_{L^2(Y)}
+ C \| \nabla w\|_{L^2(Y)} \|\nabla \widetilde{w} \|_{L^2(Y)}\\
& \le C \lm^2 \| u_0\|_{H^2(Y)} \| v_0\|_{H^2(Y)}.
\endaligned
$$
By (\ref{s-pp-3}) we obtain 
$$
 |\langle \mathcal{L}^\lm (w) , v_0 \rangle |
\le \lm^2 \|\Delta u_0\|_{L^2(Y)} \|\Delta v_0\|_{L^2(Y)}.
$$
Since $\|\nabla A\|_\infty\le L<\infty$, the $H^2$ estimates,
$\|u_0\|_{H^2(Y)} \le C \| F\|_{L^2(Y)}$ and $\| v_0\|_{H^2(Y)} \le C \|\widetilde{F}\|_{L^2(Y)}$ hold.
As a result, we have proved that
$$
\Big| \int_Y
w \cdot \widetilde{F} \, dx \Big|
\le C \lm^2 \| F\|_{L^2(Y)} \|\widetilde{F} \|_{L^2(Y)},
$$
which, by duality, gives (\ref{s-pp-10}).
\end{proof}


\section{Qualitative homogenization}\label{q-homo}

The qualitative homogenization  for the elliptic system \eqref{eq1} was established in \cite{lions1978, francfort1994} for $\kappa=\va^{\gamma}$,
where  $0<\gamma< \infty$.
  Here we consider a general case $\kappa= \kappa (\varepsilon)  $ under the condition  \eqref{ratio}.
   Denoting $\kappa \va^{-1}$ as $\lm=\lm(\va)$, the system   \eqref{eq1} may be written as
  \begin{align}\label{llam}
 \lm^2 \va^2 \Delta^2u_{\va, \lm}   -\text{div} (A(x/\va)  \na u_{\va, \lm}   )= F.
\end{align}
We first  fix $0< \lm <\infty$ and investigate the homogenization of the system (\ref{llam}).

For $1\le \beta, j \le d$, 
let $P_j^\beta = y_j (0, \cdots, 1, \dots, 0)$ with $1$ in the $\beta^{th}$ position.
Consider the cell problem,
\begin{equation}\label{chilam}
\begin{cases}
   \lambda^2 \Delta^2 \chi_j^{\lambda, \beta}  -\text{div}\big[A(y) \na (P_j^\beta + \chi_j^{\lambda, \beta}  )\big]=  0
 ~~    \text{ in } \mathbb{R}^d,\\
\chi_j^{\lambda, \beta}  (y)\, \text{ is 1-periodic in } y, \\
  \int_Y  \chi_j^{\lambda, \beta}  (y)\, dy =0,
 \end{cases}
\end{equation}
where $Y=[0,1]^d$.
Under conditions \eqref{econ} and \eqref{pcon}, for each $\lm>0$,
 \eqref{chilam} admits a unique solution 
 $\chi_j^{\lambda, \beta} =(\chi_j^{\lambda,1  \beta},...,\chi_j^{\lambda, d\beta} )$ in $H^3_{\text{loc}} (\mathbb{R}^d;\R^d)$.
 This may be proved by using the Lax-Milgram Theorem on $H^2_{\text{per}} (Y; \R^d)$.
  Moreover, let $\chi^\lambda  =\big(\chi_j^{\lambda, \alpha \beta} \big)$, then 
\begin{align}\label{echilm}
\begin{split}
\| \chi^\lm\|_{H^1(Y)} & \leq C (1+\lm)^{-2}, \\
 \|\na^2\chi^\lm\|_{L^2(Y)} & \leq C \lm^{-1} (1+\lm)^{-1} ,\\
\| \na^3 \chi^\lm\|_{L^2(Y)} & \leq C\lm^{-2},
\end{split}
\end{align}
for some constant $C$ depending only on $d$, $\nu_1$ and $\nu_2$.
Estimates in (\ref{echilm}) follow  from energy estimates. 
Indeed, by using the test functions $\chi^\lambda$ and $\Delta \chi^\lambda $ and a Korn inequality, one obtains 
$$
\lambda \| \nabla^2 \chi^\lambda\|_{L^2(Y)} 
+\| \nabla \chi^\lambda \|_{L^2(Y)} \le C,
$$
and $\|\nabla^3 \chi^\lambda\|_{L^2(Y)} \le C\lambda^{-2} $.
The remaining estimates in (\ref{echilm}) follow readily  by Poincar\'e's inequality.
If $\lambda=0$, it is well known that (\ref{chilam}) has a unique solution in $H^1_{\text{loc}} (\mathbb{R}^d)$ and
$\|\chi^0\|_{H^1(Y)} \le C$.

Thanks to \cite{lions1978}, for each fixed $\lm\ge 0$,  the homogenized operator of $ \mathcal{L}^\lambda_\va$ in (\ref{lvalm})   is given by
\begin{align}\label{l0lm}
\mathcal{L}^\lambda_0= -\text{div} (\widehat{A^\lambda } \na ),
\end{align}
where
\begin{align}\label{alam}
\widehat{A^\lambda  } =\fint_Y [A (y)+A(y)\na\chi^\lambda (y)] \, dy.
\end{align}
In view of (\ref{echilm}), we have
  $ |\widehat{A^\lambda }| \leq C $, where  $C$ depends only on $d$, $\nu_1$ and $\nu_2$. 
  
  \begin{lemma}\label{ela-lemma}
  The constant matrix $\widehat{A^\lambda}$ satisfies the elasticity condition \eqref{econ}  with the same $\nu_1$ and $\nu_2$.
  \end{lemma}
  
  \begin{proof}
  Let $\widehat{A^\lambda} = \big( \widehat{A^\lambda}_{ij}^{\alpha\beta} \big)$ with $1\le \alpha, \beta, i, j \le d$. Note that 
  $$
  \aligned
  \widehat{A^\lambda}_{ij}^{\alpha\beta}
   & =\fint_Y
  A_{ik}^{\alpha \gamma} \frac{\partial}{\partial y_k}
  \left[ P_j^{\gamma \beta} + \chi_j^{\lambda, \gamma \beta} \right]\, dy\\
  &=\fint_Y 
  A_{\ell k}^{t\gamma}
  \frac{\partial}{\partial y_k}
  \left[ P_j^{\gamma \beta} +\chi_j^{\lambda, \gamma \beta} \right]
  \cdot
  \frac{\partial}{\partial y_\ell}
  \left[ P_i^{t \alpha} + \chi_i^{\lambda, t \alpha} \right]\, dy
  +\lambda^2
  \fint_Y \Delta \chi_j^{\lambda, \beta} \cdot \Delta \chi_i^{\lambda, \alpha} \, dy,
  \endaligned
  $$
  where $P_j^{\gamma\beta} =y_j \delta^{\gamma \beta}$ and
  we have used (\ref{chilam}) for the last step.
  It follows that $\widehat{A^\lambda}$ satisfies the symmetry conditions in (\ref{econ}).
  To prove the ellipticity condition in  (\ref{econ}),  
  we introduce  the bilinear form,
  $$
  a_{\text{per}} (\phi, \psi)
  =\fint_Y A\nabla \phi \cdot \nabla \psi \, dy +\lambda^2 \fint_Y \Delta \phi \cdot \Delta \psi\, dy,
  $$
  which is symmetric and nonnegative. It is known  that the elasticity condition (\ref{econ}) implies 
  \begin{equation}\label{elas-2}
  \frac{\nu_1}{4}
  |\zeta +\zeta^T|^2
  \le A\zeta \cdot \zeta^T \le \frac{\nu_2}{4} 
  |\zeta +\zeta^T|^2
  \end{equation}
  for any  matrix $\zeta \in \mathbb{R}^{d\times d}$,
  where $\zeta^T$ denotes the transpose of $\zeta$.
  Let $\xi=(\xi_j^\beta)\in \mathbb{R}^{d\times d}$ be a symmetric matrix.
  Let 
$\phi= \xi_j^\beta P_j^\beta $ and $\psi =\xi_j^\beta \chi_j^{\lambda, \beta}$.
Then
$$
\aligned
\widehat{A^\lambda}_{ij}^{\alpha\beta} \xi_i^\alpha\xi_j^\beta
& =a_{\text{per}} (\phi +\psi, \phi +\psi)\\
&\ge \fint_Y A\nabla (\phi +\psi)\cdot \nabla (\phi +\psi)\, dy\\
&\ge \frac{\nu_1}{4}
\fint_Y  |\nabla \phi + \nabla \psi 
+ (\nabla \phi)^T 
+ (\nabla \psi)^T|^2\, dy\\
&=\frac{\nu_1}{4}
\fint_Y
|\nabla \phi + (\nabla \phi)^T|^2\, dy
+\frac{\nu_1}{4}
\fint_Y |\nabla \psi + (\nabla \psi)^T|^2\, dy\\
 & \ge  \nu_1 |\xi|^2,
\endaligned
$$
where we have used (\ref{elas-2}) and  the fact $\int_Y\nabla  \chi^\lambda\, dy=0$.
Also, note that
$$
\aligned
\widehat{A^\lambda}_{ij}^{\alpha\beta} \xi_i^\alpha\xi_j^\beta
& = a_{\text{per}} (\phi +\psi, \phi -\psi)\\
&= a_{\text{per}} (\phi, \phi) -a_{\text{per}} (\psi, \psi)
\le a_{\text{per}} (\phi, \phi) 
 \le \nu_2 |\xi|^2,
\endaligned
$$
where we have used (\ref{elas-2}) for the last inequality.
  \end{proof}
 
 Define
\begin{equation}\label{A-bar}
\overline{A} =\fint_Y  A(y)\, dy.
\end{equation}

 \begin{lemma}\label{ith21}
 Assume  $ A$ satisfies \eqref{econ} and \eqref{pcon}. 
  Let $\widehat{A^\lambda}$ be defined by \eqref{alam}.
  Then \begin{align}
 \big|\widehat{A^\lambda }- \overline{A}  \big| 
  & \leq C \lm^{-2}  \quad\text{ for } 1\le \lm<\infty,\label{reith21}\\
 \big|\widehat{A^{\lambda_1} }-\widehat{A^{\lambda_2} }  \big| & \leq C | 1- (\lambda_1/\lambda_2)^2 | \quad  \text{ for  }
 0< \lambda_1, \lambda_2<\infty, \label{reith23} \\
  \big|\widehat{A^\lambda }-\widehat{A^0} \big|
   & \leq
   \widetilde{C} \lm ^2   \quad\text{ for  }  0< \lambda\leq 1,  \text{ if in addition } \|\na A\|_{\infty}\le L, \label{reith212}
 \end{align}
 where $C$ depends only on $d$, $\nu_1$, $\nu_2$, and $\widetilde{C}$ depends on $d$, $\nu_1$, $\nu_2$ and $L$.
 \end{lemma}
 
\begin{proof}

By  the definitions of $\widehat{A^\lm}$ and $\overline{A}$,
\begin{align*}
\big|\widehat{A^\lambda } -\overline{A}
    \big|=\big|\fint_Y A(y) \na \chi^\lm (y) dy\big|\leq C\|\na  \chi^\lm\|_{L^2(Y)},
\end{align*}
which, together with \eqref{echilm}, gives \eqref{reith21}. Similarly, by the definition of $\widehat{A^\lambda}$ ,
\begin{align}\label{alm-wa}
\big|\widehat{A^{\lambda_1}  }-\widehat{A^{\lambda_2} }  \big|\leq C \|\na \chi^{\lambda_1} -\na \chi^{\lm_2} \|_{L^2(Y)}.
\end{align}
Since
\begin{align*}
-\text{div}\big(A(y) \na (\chi^{\lambda_1} -\chi^{\lambda_2} )\big)+
\lm_1^2\Delta^2 (\chi^{\lambda_1} -\chi^{\lambda_2} ) = (\lambda_2^2-\lm_1^2)\Delta^2 \chi^{\lm_2},
\end{align*}
by energy estimates and the $H^3$ estimate for $\chi^\lm$ in  \eqref{echilm},
\begin{align*}
\aligned
\|\na \chi^{\lambda_1} -\na \chi^{\lm_2} \|_{L^2(Y)} & \leq
C |\lambda_2^2 -\lambda_1^2| \| \nabla^3 \chi^{\lambda_2}\|_{L^2(Y)}\\
&\le C | 1- (\lambda_1/\lambda_2)^2|,
\endaligned
\end{align*}
which, combined with \eqref{alm-wa}, gives \eqref{reith23}.

We now turn to \eqref{reith212}. Note that
\begin{align}\label{pith212}
\big|\widehat{A^\lambda } - \widehat{A^0}\big| =\big|\fint_Y A (\na \chi^\lm -\na \chi^0 )  dy\big|
    \leq C \|\nabla A\|_\infty 
      \| \chi^\lm-\chi^0\|_{L^2(Y)},
\end{align}
where we have used the integration by parts for the last inequality.
It follows by Theorem \ref{s-pp-theorem} that 
\begin{align}\label{pith213}
 \| \chi^\lm  -\chi^0\|_{L^2(Y)}\leq C \lambda^2,
\end{align}
 where $C$ depends only on $d$, $\nu_1$, $\nu_2$ and $L$.
  This, combined with \eqref{pith212}, gives \eqref{reith212}.
\end{proof}

 Define $\mathcal{L}_0 =-\text{\rm div} (\widehat{A}\nabla)$, where 
 \begin{equation}\label{ahat}
 \widehat{A} =
 \left\{
 \aligned
& \overline{A}=\fint_Y A (y)\, dy & \quad & \text{ if }  \rho=\infty,\\
& \widehat{A^\rho} & \quad & \text{ if } 0\le \rho <\infty,
 \endaligned
 \right.
\end{equation}
where $\widehat{A^\rho}$ is given by (\ref{alam}).

\begin{lemma}\label{con-lemma}
Suppose that $\lambda \to \rho$.
Then $\widehat{A^\lambda} \to \widehat{A}$.
\end{lemma}

\begin{proof}
In view of Lemma \ref{ith21}, this is obvious if $0< \rho \le \infty$.
In the case $\rho=0$, where $\widehat{A}=\widehat{A^0}$,
 the estimate (\ref{reith212}) requires  that $A$ is  Lipschitz continuous.
The condition may be removed by an approximation argument.
Indeed, let $B$ be a smooth matrix satisfying \eqref{econ}-\eqref{pcon}.
Then
\begin{align} \label{pith214}
 \big|\widehat{A^\lambda } -
  \widehat{A^0} |
  \leq
   \big|\widehat{A^\lambda } -
   \widehat{B^\lambda }\big|
   +\big|\widehat{B^\lambda } -
    \widehat {B^0}
    \big|
    +\big| \widehat{B^0}-
   \widehat{A^0} \big|.
\end{align}
Let $\tau^\lambda$ be the weak solution of the cell problem \eqref{chilam} with $A$ being replaced by $B$. Then
\begin{align*}
\lm^2 \De^2(\chi^\lm- \tau^\lambda)-\text{div} ( A(y)\na ( \chi^\lm-{\tau^\lambda}) )= \text{div} \big(( A -B ) \na (y +\tau^\lambda)\big).
\end{align*}
By the reverse H\"older estimate (\ref{RH-1}), there exist some $p>2$ and $C>0$,
 depending only on $d$, $\nu_1$ and $\nu_2$, such that $\|\nabla \tau^\lambda\|_{L^p(Y)}\le C$.
By energy estimates,
\begin{align*}
\|\na (\chi^\lm- \tau^\lambda )\|_{L^2(Y)}&\leq C \|A-B\|_{L^2(Y)} +C\Big(\fint_Y |A-B|^2 |\na \tau^\lambda|^2dy\Big)^{1/2}\nonumber\\
&\leq C \|A-B\|_{L^2(Y)} + C \|A-B\|_{L^{q}(Y)} \|\na \tau^\lambda\|_{L^p(Y)}\nonumber\\
&\leq C  \|A-B\|_{L^{q}(Y)},
\end{align*}
where $q=2p/(p-2)$.
By the definitions of $\widehat{A^\lambda }$ and $\widehat{B^\lambda }$, we  obtain that
\begin{align}\label{pith215}
\big|\widehat{A^\lambda }-\widehat{B^\lambda }\big|\leq \|A-B\|_{L^2(Y)}+ \|\na (\chi^\lm-\tau^\lambda) \|_{L^2(Y)}\leq C   \|A-B\|_{L^{q}(Y)}.
\end{align}
Similarly,  one can prove that
\begin{align*}
\big| \widehat{A^0} - \widehat{B^0}\big|
\leq \|A-B\|_{L^2(Y)}+ \|\na (\chi^0-\tau^0)\|_{L^2(Y)}\leq C   \|A-B\|_{L^{q}(Y)},
\end{align*}
which, combined with \eqref{pith214}, \eqref{pith215} and \eqref{reith212} for $B$, gives
 \begin{align*}
 \big|\widehat{A^\lambda} -
   \widehat{A^0} \big|\leq C  \|A-B\|_{L^{q}(Y)}+C_B  \lm^2,
 \end{align*}
 where $C_B$ depends on $\|\nabla B\|_\infty$. 
 By approximating $A$ in $L^q(Y) $ with a sequence of smooth matrix satisfying \eqref{econ} and \eqref{pcon},
 we obtain $\widehat{A^\lambda} \to \widehat{A^0}$ as $\lambda\to 0$.
 \end{proof}
 
The following theorem shows that the effective  equation   for   \eqref{eq1} is given by $\mathcal{L}_0 (u_0)=F$.

\begin{theorem}
Suppose that $A$ satisfies \eqref{econ}-\eqref{pcon} and $\kappa$ satisfies \eqref{ratio}. 
Let $F\in H^{-1}(\Omega; \R^d)$ and $G\in H^2(\Omega; \R^d)$, where $\Omega$ is a bounded Lipschitz domain in $\R^d$.
Let $u_\va\in H^2(\Omega; \R^d) $ be the weak solution of  \eqref{eq1} such that $u_\va-G\in H_0^2(\Omega;\R^d)$.
Let $u_0\in H^1(\Omega; \R^d) $ be the weak solution
of $-\text{\rm div}(\widehat{A}\nabla u_0)=F$ in $\Omega$ with $u_0-G \in H_0^1(\Omega; \R^d)$, where  $\widehat{A}$ is given by \eqref{ahat}. 
Then as $\va\rightarrow 0$, $u_\va \to u_0$ weakly in $H^1(\Omega; \R^d)$, and
$A(x/\va)\nabla u_\va \to \widehat{A}\nabla u_0$ weakly in $L^2(\Omega; \R^{d\times d} )$.
\end{theorem}

\begin{proof}
This is  proved by using Tartar's  method of test functions.
Note that since $\kappa< 1$,
 by the energy estimate (\ref{energy-0}), 
\begin{equation}\label{e-estimate-1}
\kappa \|\nabla^2 u_\va\|_{L^2(\Omega)} +
 \| u_\va\|_{H^1(\Omega)} \le C\big\{  \| F \|_{H^{-1}(\Omega)} +  \| G\|_{H^2(\Omega)} \big\},
\end{equation}
where $C$ depends on $d$, $\nu_1$, $\nu_2$ and $\Omega$.
Let $\{ u_{\va^\prime}\}$ be a sequence such that 
$u_{\va^\prime} \to u$ weakly in $H^1(\Omega; \R^d)$ and
$A(x/\va^\prime) \nabla u_{\va^\prime} \to H$ weakly in $L^2(\Omega; \R^{d\times d} )$.
We will show that $H=\widehat{A}\nabla u$ in $\Omega$.
Since $-\text{\rm div}(H)=F$ in $\Omega$, we see that 
$-\text{\rm div} (\widehat{A}\nabla u)=F$ in $\Omega$.
By the uniqueness of weak solutions in $H^1(\Omega; \R^d)$ for $\mathcal{L}_0$,
we deduce that $u=u_0$.
As a result, we obtain that $u_\va \to u_0$ weakly in $H^1(\Omega; \R^d)$ and
$A(x/\va)\nabla u_\va \to \widehat{A}\nabla u_0$ weakly in $L^2(\Omega; \R^{d\times d } )$, as $\va\to 0$.

To show $H=\widehat{A}\nabla u$, for notational simplicity, we let $\va=\va^\prime$
and  $\lambda =\kappa/\va $.
Note that
\begin{equation}\label{cor-eq}
\mathcal{L}_\va\Big\{
P_j^\beta + \va \chi_j^{\lambda,\beta} (x/\va) \Big\} =0 \quad \text{ in } \mathbb{R}^d.
\end{equation}
It follows that
\begin{equation}\label{h-10}
\aligned
& \lambda^2 \va^2\int_\Omega
\Delta \big\{ P_j^\beta +\va \chi^{\lambda, \beta}_j (x/\va) \big\}
\cdot \Delta (u_\va \psi)\, dx\\
&\qquad\qquad
+
\int_\Omega
A(x/\va) \nabla \big( P_j^\beta +\va \chi_j^{\lambda, \beta} (x/\va) \big)
\cdot \nabla (u_\va \psi) \, dx=0,
\endaligned
\end{equation}
for any $\psi\in C_0^\infty (\Omega)$.
Also note that
\begin{equation}\label{h-11}
\aligned
& \lambda^2 \va^2\int_\Omega
\Delta u_\va \cdot \Delta \big\{ ( P_j^\beta +\va \chi_j^{\lambda, \beta} (x/\va) ) \psi \big\}\, dx\\
& +
\int_\Omega A(x/\va)\nabla u_\va
\cdot
\nabla \big\{ (P_j^\beta +\va \chi_j^{\lambda, \beta} (x/\va) )  \psi \big\}\, dx
= \langle F, ( P_j^\beta +\va \chi_j^{\lambda, \beta} (x/\va) ) \psi \rangle.
\endaligned
\end{equation}
By subtracting (\ref{h-10}) from (\ref{h-11}), we obtain 
\begin{equation}\label{h-12}
\aligned
& 2\lambda^2 \va^2 \int_\Omega
\Delta u_\va \cdot \nabla (P_j^\beta +\va \chi_j^{\lambda, \beta} (x/\va) ) \nabla \psi\, dx\\
& \qquad -2\lambda^2 \va^2 \int_\Omega
\Delta ( P_j^\beta +\va \chi_j^{\lambda, \beta} (x/\va) ) \cdot \nabla u_\va \cdot \nabla \psi\, dx\\
&
\qquad +\lambda^2 \va^2\int_\Omega 
\Delta u_\va \cdot (P_j^\beta +\va \chi_j^{\lambda, \beta} (x/\va) ) \Delta \psi\, dx\\
&\qquad -\lambda^2 \va^2\int_\Omega
\Delta ( P_j^\beta +\va \chi_j^{\lambda, \beta} (x/\va) ) u_\va \cdot \Delta \psi\, dx\\
&\qquad + 
\int_\Omega A(x/\va) \nabla u_\va \cdot ( P_j^\beta +\va \chi_j^{\lambda, \beta} (x/\va) ) \nabla \psi\, dx\\
& \qquad
-\int_\Omega A(x/\va) \nabla ( P_j^\beta +\va \chi_j^{\lambda, \beta} (x/\va) ) u_\va \cdot \nabla \psi\, dx\\
&=  \langle F, ( P_j^\beta +\va \chi_j^{\lambda, \beta} (x/\va) ) \psi \rangle.
\endaligned
\end{equation}
We now let $\va\to 0$ in (\ref{h-12}).
Using (\ref{e-estimate-1}) and (\ref{echilm}), it is not hard to see that the first four terms in the left-hand side of
(\ref{h-12}) converge  to zero, while the right-hand side converges to $\langle F, P_j^\beta \psi\rangle$.
Also, the fifth term in the left-hand side converges to 
$$
\int_\Omega H_i^\alpha  \cdot P_j^{\alpha \beta } \frac{\partial \psi}{\partial x_i}  \, dx
=\langle F, P_j^\beta \psi \rangle
-\int_\Omega H_j^\beta \psi\,dx.
$$
Finally, we observe that by Lemma \ref{con-lemma},
 $\widehat{A^\lambda} \to \widehat{A}$ as $\va\to 0$, and that $u_\va \to u$ strongly in $L^2(\Omega; \R^d)$.
This implies that the last term in the left-hand side of (\ref{h-12}) converges to
$$
-\int_\Omega \widehat{A}_{ij}^{\alpha\beta} u^\alpha \frac{\partial \psi}{\partial x_i}\, dx
=\int_\Omega \widehat{A}_{ij}^{\alpha\beta} \frac{\partial u^\alpha}{\partial x_i}\psi\, dx,
$$
where we have used integration by parts.
Since $\psi \in C_0^\infty (\Omega)$ is arbitrary,
we see that
$$
H_j^\beta = \widehat{A}_{ij}^{\alpha\beta} \frac{\partial u^\alpha}{\partial x_i}
= \widehat{A}_{ji}^{\beta\alpha } \frac{\partial u^\alpha}{\partial x_i},
$$
where we have used the symmetry conditions of $\widehat{A}$.
Hence, $H=\widehat{A}\nabla u$.
\end{proof}


\section{Convergence rates}\label{section-4}

In this section we give the proof of Theorem \ref{coth1}.
To this end, 
we   fix $0<\lambda<\infty$ and consider the Dirichlet problem, 
\begin{equation}\label{DP-3}
\mathcal{L}_\va^\lm (u_{\va, \lm} )  =F \quad \text{ in } \Omega
\quad \text{ and } \quad
u_{\va, \lm}-G \in H^2_0(\Omega; \R^d),
\end{equation}
where  $\mathcal{L}_\va^\lm$ is given by (\ref{lvalm}),
 $F\in L^2(\Omega; \R^d)$ and $G\in H^2(\Omega; \R^d)$.
 Let $u_{0, \lm}\in H^1(\Omega; \R^d)$ be the solution of the homogenized problem,
\begin{equation}\label{homo-4}
-\text{\rm div} \big( \widehat{A^\lm}  \nabla u_{0, \lm} \big) =F 
\quad \text{ in } \Omega \quad \text{ and } \quad u_{0, \lm}-G \in H_0^1(\Omega; \R^d),
\end{equation}
where $\widehat{A^\lm}$ is given by (\ref{alam}).
We shall study the convergence rate of $u_{\va,  \lm}$ to $u_{0, \lm}$ as $\va \to 0$.

Let $\eta_{t} \in C_{0}^{\infty}(\Omega)$ be a cut-off function such that
 \begin{align}\label{cutoff}
\begin{split}
&0\leq \eta_t \leq 1,  |\nabla^k \eta_t| \leq C t^{-k} \text{ for }  k=1,2,\\
&\eta_t=1 ~\text{ if } x\in \Omega\!\setminus\!\Omega_{4t} \quad\text{and}\quad
 \eta_{t}(x)=0  ~\text{ if } x \in \Omega_{3t}, \end{split}
 \end{align}
where $\va \le t<1$ and $ \Omega_t $ is defined in  \eqref{omegava}.
Let 
\begin{equation}\label{w-e-1}
w_{\va, \lm}
= u_{\va, \lm} - u_{0, \lm} + (u_{0, \lm} -G) (1-\eta_t) 
- \va \chi^\lm (x/\va) \eta_t  S_\va (\nabla u_{0, \lm} ),
\end{equation}
where $t=(1+\lm)\va $ and $\chi^\lambda$ is the corrector given by (\ref{chilam}).
The $\va$-smoothing operator $S_\va$ in (\ref{w-e-1}) is  defined by
\begin{align*}
  S_\varepsilon(f)(x)= \int_{\mathbb{R}^d}f(x-\varsigma )\varphi_\varepsilon(\varsigma) d\varsigma,
\end{align*}
 where $\varphi_{\va }(\varsigma)=\varepsilon^{- d}\varphi(\varsigma / \va)$ and
  $\varphi$ is a fixed function in $ C_0^{\infty}(B(0,1/2))$ such that $\varphi\geq 0$ and $\int_{\mathbb{R}^{d}}\varphi dx=1$.

\begin{lemma} \label{lsmooth1}
Let  $f\in W^{1, p}(\mathbb{R}^d)$ for some $1\leq p\leq \infty$. Then
  \begin{align}
  \|S_{\va}(f)-f\|_{L^p( \mathbb{R}^d)}\leq \va \|\nabla f\|_{L^p(\mathbb{R}^d)}. \label{lsmooth1re1} 
  \end{align}
Suppose that $f, g\in L^p_{loc}(\R^d) $ for some $1\le p<\infty$ and $ g$ is 1-periodic. Then
\begin{align}
&\| g^\va \nabla^k S_\varepsilon( f)\|_{L^p(\mathcal{O})}\leq C_k\va^{-k}   \|g \|_{L^p(Y)}\|f\|_{L^p(\mathcal{O} ^\varepsilon)} \label{22}
 \end{align}
 for $k\ge 0$, 
 where  $g^\va (x)=g(x/\va)$,
 $\mathcal{O}^\va =\{ x\in \mathbb{R}^d: \text{\rm dist} (x, \mathcal{O} )< \va\}$,  and
 $C_k $ depends only on $d$, $k$ and $p$.
\end{lemma}
\begin{proof}
See e.g., \cite{shenapde2017}.
\end{proof}

\begin{lemma}\label{lemma-3.1}
Let $\Omega$ be a bounded Lipschitz domain in $\mathbb{R}^d$.
Let $u_{\va, \lm}$, $u_{0, \lm}$ and $w_{\va, \lm}$ be given by \eqref{DP-3}, \eqref{homo-4} and \eqref{w-e-1}, respectively.
Suppose $u_{0, \lm} \in H^2(\Omega; \R^d)$.
Then for any $\psi\in H^2_0(\Omega; \R^d)$ and $0<\va<(1+\lm)^{-1}$,
\begin{equation}\label{3.1-0}
\aligned
|\langle \mathcal{L}_{\va}^\lambda  ( w_{\va, \lm} ), \psi \rangle |
 & \le C \| u_{0, \lm} \|_{H^2(\Omega)} \Big\{
\va   \|\nabla \psi\|_{L^2(\Omega)}
+ \va^2 \lm^2  \| \Delta \psi\|_{L^2(\Omega)} \Big\} \\
&\qquad
+C t^{1/2} \Big\{ \| u_{0, \lm} \|_{H^2(\Omega)} + \|G\|_{H^2(\Omega)} \Big\} \| \nabla \psi\|_{L^2(\Omega_{5t})}\\
& \qquad + 
C \va^2 \lambda^2 t^{-{1/2}} 
\Big\{ \| u_{0, \lm} \|_{H^2(\Omega)} + \|G\|_{H^2(\Omega)} \Big\} \| \Delta \psi\|_{L^2(\Omega_{5t})},
\endaligned
\end{equation}
where $t=(1+\lm)\va$ and $C$ depends only on $d$, $\nu_1$, $\nu_2$, and $\Omega$.
\end{lemma}

\begin{proof}
Note that $w_{\va, \lm} \in H^2_0(\Omega; \R^d)$ and
$$
\aligned
\mathcal{L}_{\va}^\lm  (w_{\va, \lm})
&=-\text{\rm div} \big\{   ( \widehat{A^\lm} -A (x/\va) ) \nabla u_{0, \lm} \big\} 
-\lm^2 \va^2 \Delta^2 (u_{0, \lm}) 
+\mathcal{L}_\va^\lm \big\{ (u_{0, \lm} -G) (1-\eta_t) \big\}\\
&\qquad
-\mathcal{L}_\va^\lm \big\{ \va \chi^\lm (x/\va)  \eta_t  S_\va (\nabla u_{0, \lm}) \big\}\\
&=-\text{\rm div} \big\{  ( \widehat{A^\lm} -A(x/\va ) )  ( \nabla u_{0, \lm} -\eta_t  S_\va (\nabla u_{0, \lm}) ) \big\}\\
&\qquad
-\lm^2 \va^2 \Delta^2 (u_{0, \lm}) 
+\mathcal{L}_\va^\lm \big\{ (u_{0, \lm} -G) (1-\eta_t) \big\}\\
&\qquad
-\text{\rm div}
\big\{ B^\lm (x/\va) \eta_t  S_\va (\nabla u_{0, \lm} ) \big\}
-\lm^2 \va\,  \text{\rm div}
\big\{ \Delta \chi^\lm (x/\va) \nabla [ \eta_ t S_\va (\nabla u_{0, \lm})] \big\}\\
&\qquad
-2\lm^2 \va^2 \Delta \big\{ \nabla \chi^\lm (x/\va) 
\nabla [ \eta_t S_\va (\nabla u_{0, \lm})  ] \big\}
-\lm^2 \va^3
\Delta \big\{ \chi^\lm (x/\va) \Delta [ \eta_t S_\va (\nabla u_{0, \lm})] \big\}\\
&\qquad
+\va \, \text{\rm div}
\big\{ \chi^\lm (x/\va) A(x/\va) \nabla [ \eta_t S_\va( \nabla u_{0, \lm}  ) ]  \big\},
\endaligned
$$
where
\begin{equation}\label{B}
B^\lm (y)  =\lm^2 \nabla \Delta \chi^\lm (y) - A\nabla \chi^\lm (y)  -A (y)  + \widehat{A^\lm}.
\end{equation}
It follows that for any $\psi \in H_0^2(\Omega; \R^d)$,
$$
\aligned
|\langle \mathcal{L}_{\va}^\lambda  ( w_{\va, \lm} ) , \psi \rangle |
 & \le  C \int_\Omega | [ \nabla u_{0, \lm}  -\eta_t S_\va (\nabla u_{0, \lm} )] \nabla \psi|\, dx
+ \va^2 \lm^2  \int_\Omega
| \Delta u_{0, \lm} | |\Delta \psi|\, dx\\
&\qquad
+ \lm^2 \va^2\int_\Omega
|\Delta [ (u_{0, \lm} -G ) (1-\eta_t) ] | |\Delta \psi|\, dx  \\
&\qquad
+ C \int_\Omega |\nabla [ (u_{0, \lm} -G)(1-\eta_t)] | |\nabla \psi|\, dx \\
& \qquad
+ C \Big|\int_\Omega B^\lm (x/\va) \eta_t  S_\va (\nabla u_{0, \lm} ) \nabla \psi \, dx\Big|\\
&\qquad
 +C\va
\int_\Omega |\chi^\lambda (x/\va) \nabla \big [ \eta_t S_\va (\nabla u_{0, \lm})\big ] \nabla \psi|\, dx\\
&\qquad
 + C\va\lm^2 \int_\Omega
| \nabla^2 \chi^\lambda (x/\va) \nabla \big [ \eta_t S_\va (\nabla u_{0, \lm} ) \big] \nabla \psi| \, dx \\
&\qquad
+ 
C \va^2 \lm^2 \
\int_\Omega
|\nabla \chi^\lm (x/\va) \nabla^2 \big[ \eta_t S_\va (\nabla u_{0, \lm}) \big] \nabla  \psi|\, dx \\
&\qquad 
+C \va^3 \lm^2
\int_\Omega
|   \chi^\lm (x/\va) \nabla^3 \big [ \eta_t S_\va (\nabla u_{0, \lm}) \big] \nabla  \psi|\, dx\\
& =I_1+I_2+\cdots + I_9.
\endaligned
$$
Using Lemma \ref{lsmooth1} and the Cauchy  inequality, it is not hard  to see that 
\begin{equation}\label{3.1-1}
I_1\le 
C \Big\{
\|\nabla u_{0, \lm}\|_{L^2(\Omega_{5t })}
\|\nabla \psi\|_{L^2(\Omega_{5t})}
+ \va \|\nabla^2 u_{0, \lambda}\|_{L^2(\Omega\setminus \Omega_{2t})}
\|\nabla \psi\|_{L^2(\Omega)} \Big\}.
\end{equation}
Next, we observe  that
\begin{equation}\label{3.1-1-1}
\aligned
I_2  + I_3 +I_4   & \le  \va^2 \lm^2 \| \Delta u_{0, \lm} \|_{L^2(\Omega)} \|\Delta \psi\|_{L^2(\Omega)}
+ C \lm^2 \va^2 t^{-1/2} 
\| u_{0, \lm} -G\|_{H^2(\Omega)} \|\Delta \psi \|_{L^2(\Omega_{4t})}\\
& \qquad
+ C t^{1/2} \| u_{0, \lm} -G \|_{H^2(\Omega)} \|\nabla \psi\|_{L^2(\Omega_{4t})}.
\endaligned
\end{equation}
To bound $I_5$, we note that 
by (\ref{echilm}), we have  $\|B^\lambda\|_{L^2(Y)}\le C$, where $C$ depends only on $d$, $\nu_1$ and $\nu_2$.
Moreover, by the definition of $B^\lm=(B^\lm_{ij}), 1\leq i,j\leq d$,
\begin{align}\label{eblm}
\partial_{y_i}B^\lm_{ij}=0 \quad \text{ and  } \quad
\quad \int _{Y}B^\lm_{ij}\, dy=0.
\end{align}
This allows us to
construct a matrix of 1-periodic  flux correctors $\mathfrak{B}^\lm_{kij}(y)$ such that
  $$\mathfrak{B}_{kij}^\lm= -\mathfrak{B}_{ikj}^\lm,\quad \pa_{y_k}\mathfrak{B}_{kij}^\lm(y)=B_{ij}^\lm(y), \quad
 \|\mathfrak{B}^\lm_{kij}\|_{H^1(Y)}\leq C,$$ with $C $ depending only on $d$, $\nu_1$ and $\nu_2$.
 It follows that
 \begin{equation}\label{3.1-3}
 \aligned
 I_5
 &\le C \va \| \mathfrak{B}^\lm (x/\va)  \nabla (\eta_t S_\va (\nabla u_{0, \lm})) \nabla \psi\|_{L^1(\Omega)}\\
 & \le C 
 \|\nabla u_{0, \lm} \|_{L^2(\Omega_{5t})}
 \|\nabla \psi\|_{L^2(\Omega_{5t})}
 + C \va \|\nabla^2 u_{0, \lm} \|_{L^2(\Omega\setminus \Omega_{2t})} \| \nabla \psi \|_{L^2(\Omega)},
 \endaligned
 \end{equation}
 where we have used the fact $\va t^{-1} \le 1$.
 Using (\ref{echilm}) and (\ref{22}), we also obtain 
 \begin{equation}\label{3.1-4}
 I_6+I_7  + I_8 +I_9
  \le C 
 \|\nabla u_{0, \lm} \|_{L^2(\Omega_{5t})}
 \|\nabla \psi\|_{L^2(\Omega_{5t})}
 + C \va \|\nabla^2 u_{0, \lm} \|_{L^2(\Omega\setminus \Omega_{2t})} \| \nabla \psi \|_{L^2(\Omega)}.
\end{equation}
By collecting estimates for $I_1, I_2, \dots, I_9$, 
we obtain the desired estimate (\ref{3.1-0}).
\end{proof}


\begin{lemma}\label{lemma-3.2}
Let $u_{\va, \lm}$, $u_{0, \lm}$ and $w_{\va, \lm}$ be the same as in Lemma \ref{lemma-3.1}.
Assume that $u_{0, \lambda} \in H^2(\Omega; \R^d)$.
Then
\begin{equation}\label{3.2-0}
\lm \va 
\| \Delta  w_{\va, \lm} \|_{L^2(\Omega)}
+ \|\nabla w_{\va, \lambda} \|_{L^2(\Omega)}
\le C \big(  (1+\lm) \va \big)^{1/2}  \big\{ \| u_0\|_{H^2(\Omega)}
+ \| G\|_{H^2(\Omega)} \big\},
\end{equation}
where $C$ depends only on $d$, $\nu_1$, $\nu_2$, and $\Omega$.
\end{lemma}

\begin{proof}
Note that $w_{\va, \lm}\in H_0^2(\Omega; \R^d)$. The  estimate (\ref{3.2-0}) follows readily 
by letting $\psi =w_{\va, \lm}$ in (\ref{3.1-0}) and using the Cauchy inequality
as well as the first Korn inequality.
 \end{proof}

The next theorem gives the sharp convergence rate in $L^2(\Omega)$ for
$\mathcal{L}_\va^\lm$ with $\lm$ fixed.

\begin{theorem}\label{thm-lm}
Suppose $A$ satisfies conditions \eqref{econ}  and \eqref{pcon}.
Let $\Omega$ be a bounded $C^{1, 1}$ domain,  $F\in L^2(\Omega; \R^d)$ and $G\in H^2(\Omega; \R^d)$.
Let $u_{\va, \lm}$ be the weak solution of \eqref{DP-3} and
$u_{0, \lambda}$ the solution of the homogenized problem \eqref{homo-4},
where $0<\lm <\infty$.
Then for  any $0<\va < (1+\lm)^{-1}$, 
\begin{equation}\label{3.3-0}
\| u_{\va, \lm} -u_{0, \lm}\|_{L^2(\Omega)}
\le C (1+\lambda) \va 
\big\{ \| F\|_{L^2(\Omega)} +\| G\|_{H^2(\Omega)} \big\},
\end{equation}
where $C$ depends only on $d$, $\nu_1$, $\nu_2$, and $\Omega$.
\end{theorem}

\begin{proof}
For $\widetilde{F } \in C_0^\infty(\Omega; \R^d)$, let $v_{\va, \lm}\in H^2_0(\Omega; \R^d)$ be the weak solution of
$\mathcal{L}^\lm_{\va} (v_{\va, \lm})=\widetilde{F}$ in $\Omega$ and
$v_{0, \lm}$ the solution in $H^1_0(\Omega; \R^d)$ of the homogenized problem
$-\text{\rm div} (\widehat{A^\lm} \nabla  v_{0, \lm})=\widetilde{F} $ in $\Omega$.
Note that since $\Omega$ is $C^{1, 1}$, we have
$\|v_{0, \lambda} \|_{H^2(\Omega)} \le C \| \widetilde{F} \|_{L^2(\Omega)}$
and 
$$
\|u_{0, \lm} \|_{H^2(\Omega)}
\le C \big\{ \| F\|_{L^2(\Omega)} + \| G \|_{H^2(\Omega)} \big\},
$$
where $C$ depends only on $d$, $\nu_1$, $\nu_2$, and $\Omega$.
Let
\begin{equation}\label{w-e-2}
\widetilde{w}_{\va, \lm}
=v_{\va, \lm} - v_{0, \lambda} \widetilde{\eta}_t
-\va \chi^\lm (x/\va) \widetilde{\eta}_t S_\va (\nabla v_{0, \lm}),
\end{equation}
where $ t=(1+\lm)\va$ and
$\widetilde{\eta}_t$ is a function in $C_0^\infty (\Omega)$ such that
$0\le \widetilde{\eta}_t \le 1$,
$|\nabla^k \widetilde{\eta}_t |\le C t^{-k}$ for $k=1, 2$,
 $\widetilde{\eta}_t (x)=1$ if $x\in \Omega\setminus \Omega_{8t}$, and 
$\widetilde{\eta}_t (x) =0$ if $x\in \Omega_{7t}$.

Let $w_{\va, \lm}$ be given by (\ref{w-e-1}).
Note that
\begin{equation}\label{3.3-1}
\aligned
\Big|\int_\Omega
w_{\va, \lm} \cdot \widetilde{F} \, dx \Big|
&= |\langle \mathcal{L}_\va^\lm (w_{\va, \lm}), v_{\va, \lm} \rangle |\\
&\le
 |\langle \mathcal{L}_\va^\lm (w_{\va, \lm}), \widetilde{w}_{\va, \lm} \rangle |
 + |\langle \mathcal{L}_\va^\lm (w_{\va, \lm}),  v_{0, \lm} \widetilde{\eta}_t  \rangle |
 +  |\langle \mathcal{L}_\va^\lm (w_{\va, \lm}), \zeta_{\va, \lm}  \rangle |\\
 & =J_1 + J_2+J_3,
 \endaligned
 \end{equation}
 where
 \begin{equation}\label{3.3-2}
 \zeta_{\va, \lm}= 
 \va \chi^\lm (x/\va) \widetilde{\eta}_t S_\va (\nabla v_{0, \lm}).
 \end{equation}
Observe that 
\begin{equation}\label{3.3-3}
\aligned
J_1  &\le   \va^2 \lm^2 \|\Delta w_{\va, \lm} \|_{L^2(\Omega)}
\|\Delta \widetilde{w}_{\va, \lm} \|_{L^2(\Omega)}
+ C \|\nabla w_{\va, \lm } \|_{L^2(\Omega )}
\| \nabla \widetilde{w}_{\va, \lm} \|_{L^2(\Omega)}\\
& \le C (1+\lm) \va 
\big\{ \| u_{0, \lm} \|_{H^2(\Omega)}  + \|G\|_{H^2(\Omega)}  \big\} \| v_{0, \lm} \|_{H^2(\Omega)},
\endaligned
\end{equation}
where we have used (\ref{3.2-0}) for the last inequality.
To bound $J_2$, we use (\ref{3.1-0}) to obtain 
\begin{equation}\label{3.3-4-0}
\aligned
J_2
& \le C (1+\lm) \va  \big\{  \| u_{0, \lm} \|_{H^2(\Omega)}  + \| G\|_{H^2(\Omega)} \big\} \| v_{0, \lm} \|_{H^2(\Omega)}.
\endaligned
\end{equation}
To handle $J_3$, we note that by (\ref{echilm}) and (\ref{22}),
\begin{align}
\|\nabla \zeta_{\va, \lm} \|_{L^2(\Omega)} 
 & \le C (1+\lm)^{-2}  \| v_{0, \lm}\|_{H^2(\Omega)}, \label{3.3-4}\\
 \|\Delta \zeta_{\va, \lm} \|_{L^2(\Omega)}
  & \le C \va^{-1} (1+\lm)^{-2}
 \| v_{0, \lm} \|_{H^2(\Omega)}. \label{3.3-4-1}
\end{align}
Since $\zeta_{\va, \lm} =0$ in $\Omega_{5t}$,
it follows from (\ref{3.1-0}) that
\begin{equation}\label{3.3-5}
\aligned
J_3
&\le C \va \| u_{0, \lm} \|_{H^2(\Omega)}
\|\nabla \zeta_{\va, \lm} \|_{L^2(\Omega)}
+ C \va^2  \lm^2 \| u_{0, \lm} \|_{H^2(\Omega)}
\|\Delta  \zeta_{\va, \lm} \|_{L^2(\Omega)}\\
&\le  C \va \| u_{0, \lm}\|_{H^2(\Omega)}
\| v_{0,\lm}\|_{H^2(\Omega)}.
\endaligned
\end{equation}
In view of (\ref{3.3-1}), (\ref{3.3-3}), (\ref{3.3-4-0}) and (\ref{3.3-5}), we have proved that
$$
\aligned
\Big|\int_\Omega
w_{\va, \lm} \cdot \widetilde{F} \, dx \Big|
 & \le C (1+\lm)  \va  \big\{ \| u_{0, \lm}\|_{H^2(\Omega)} + \|G\|_{H^2(\Omega) } \big\} 
\| v_{0,\lm}\|_{H^2(\Omega)}\\
& \le C (1+\lm)  \va \big\{   \| F\|_{L^2(\Omega)} + \| G\|_{H^2(\Omega)}  \big\} 
\| \widetilde{F}  \|_{L^2(\Omega)}.
\endaligned
$$
By duality this implies that
$$
\| w_{\va, \lm} \|_{L^2(\Omega)}
\le C (1+\lambda) \va  \big\{ \| F \|_{L^2(\Omega)} +  \| G \|_{H^2(\Omega)} \big\}.
$$
Hence,
$$
\aligned
\| u_{\va, \lm} -u_{0, \lm}\|_{L^2(\Omega)}
& \le \| w_{\va, \lm}  \|_{L^2(\Omega)}
+\|  (u_{0, \lm} -G ) (1-\eta_t) \|_{L^2(\Omega)}\\
&\qquad\qquad\qquad
+ \|\va \chi^\lm (x/\va) \eta_t S_\va (\nabla u_{0, \lm})\|_{L^2(\Omega)}\\
&
\le C (1+\lambda) \va
\big\{ \| F \|_{L^2(\Omega)} +  \| G \|_{H^2(\Omega)} \big\},
\endaligned
$$
which completes the proof.
\end{proof}

We are now ready  to prove Theorem \ref{coth1}.

\begin{proof}[\textbf{Proof of Theorem \ref{coth1}}] 

Let $u_\va\in H^2(\Omega; \R^d)$ be a weak solution of $\mathcal{L}_\va (u_\va)=F$ in $\Omega$ with $u_\va -G \in H_0^2(\Omega)$, and
$u_0\in H^1 (\Omega; \R^d)$ the solution of the homogenized equation $-\text{\rm div}(\widehat{A} \nabla u_0)=F$ in $\Omega$
with $u_0-G \in H_0^1(\Omega; \R^d)$.
Let $\lm=\kappa/\va$. Then $\mathcal{L}_\va^\lm (u_\va) = \mathcal{L}_\va (u_\va)=F$ in $\Omega$.
Let $u_{0, \lm}\in H^1(\Omega; \R^d)$ be the solution of  $-\text{\rm div}(\widehat{A^\lambda} \nabla u_{0, \lm})=F$ in $\Omega$
with $u_{0, \lm}-G \in H_0^1(\Omega; \R^d)$.
Note that
\begin{equation}\label{3-10-0}
\aligned
\| u_\va -u_0\|_{L^2(\Omega)}
 & \le \| u_\va - u_{0, \lm}\|_{L^2(\Omega)}
+\| u_{0, \lm} -u_0\|_{L^2(\Omega)}\\
& \le 
C  (\kappa +\va) \big\{ \| F\|_{L^2(\Omega)} +\|G\|_{H^2(\Omega)} \big\} 
+ \| u_{0, \lm} -u_0\|_{L^2(\Omega)},
\endaligned
\end{equation}
where we have used Theorem \ref{thm-lm} for the last inequality.
To estimate $u_{0, \lm} -u_0$, we observe that $u_{0, \lm} -u_0\in H_0^1(\Omega; \R^d)$ and
$$
-\text{\rm div} (\widehat{A} \nabla (u_0 -u_{0, \lambda}))
=\text{\rm div} ( (\widehat{A} -\widehat{A^\lm})\nabla u_{0, \lm})
$$
in $\Omega$.
By energy estimates,
$$
\aligned
\|u_0-u_{0, \lm}\|_{H^1(\Omega)}
 & \le C  |\widehat{A} -\widehat{A^\lm}| \| \nabla u_{0, \lm}\|_{L^2(\Omega)}\\
& \le C   |\widehat{A} -\widehat{A^\lm}|  \big\{ \|  F \|_{L^2(\Omega)}
+ \| G \|_{H^1(\Omega)} \big\} ,
\endaligned
$$
where $C$ depends only on $d$, $\nu_1$,  $\nu_2$, and $\Omega$.
This, together with Lemma \ref{ith21} and (\ref{3-10-0}), gives (\ref{re1}).
\end{proof}


\section{Approximation}\label{section-app}

Fix $0<\lm<\infty$.
Let $\mathcal{L}^\lm_{\va}$  be defined as in \eqref{llam}.
The goal of this section is to establish the following.

\begin{theorem}\label{apth2}
 Suppose $A$ satisfies \eqref{econ}  and  \eqref{pcon}. 
 Let $u_{\va, \lm} \in H^2(B_{2r}; \R^d )$ be a solution to $\mathcal{L}^\lm_{\va}   ( u_{\va, \lm}) =F$ in $B_{2r}$, 
 where $F\in L^p(B_{2r}; \R^d)$ and $B_{2r} =B(z, 2r)$ for some $z\in \mathbb{R}^d$.
 Assume  that $p>d$ and $\va\le  r<\infty$.
  Then there exists  $v_{\va , \lm}\in H^2(B_r; \R^d)$ such that  
  \begin{equation}\label{L-0-lm}
  \va^2 \lm^2 \Delta^2 v_{\va, \lm} -\text{\rm div} \big( \widehat{A^\lm} \nabla v_{\va, \lm}\big) =F \quad \text{ in } B_r,
  \end{equation}
  and
   \begin{align}
&  \qquad   \left(\fint_{B_r} |\nabla v_{\va, \lm} |^2 \right)^{1/2}
   \le C \left(\fint_{B_{2r}} |\nabla u_{\va, \lm}|^2 \right)^{1/2}, \label{s-ap-0}\\
    &\left (\fint_{B_r}|\na u_{\va, \lm }-\na v_{\va, \lm}-(\na\chi^\lm)(x/\va) \na v_{\va, \lm}|^2\right)^{1/2}\nonumber\\
    &\leq C \left(\frac{\va }{r}\right)^{^\sigma} \left\{ \left(\fint_{B_{2r}}|\na u_{\va, \lm} |^2\right)^{1/2}+r \left(\fint_{B_{2r}}|F|^p\right)^{1/p}\right\},\label{apth2re1}
  \end{align}
  where $C>0$ and $0<\sigma<1$ depend only on $d$, $\nu_1 $, $\nu_2$, and $p$.
\end{theorem}

To prove Theorem \ref{apth2},  we introduce  an intermediate  Dirichlet problem,
\begin{align}\label{eqv0}
\lm^2  \varepsilon^{2}\Delta^{2}v_{\va, \lm}   + \mathcal{L}^\lm_0 ( v_{\va, \lm} )   =F \quad 
\text{ in } \Omega \quad \text{ and } \quad
 v_{\va, \lm} -G \in H^2_0(\Omega; \R^d),
\end{align}
where $\mathcal{L}_0^\lm =-\text{\rm div} (\widehat{A^\lm}\nabla )$ and $\widehat{A^\lm}$ is defined  by (\ref{alam}).
We will establish a (suboptimal) convergence rate in $H^1(\Omega)$ for $u_{\va, \lm} - v_{\va, \lm}$, where  $u_{\va, \lm}$ is the solution to 
the Dirichlet problem,
\begin{equation}\label{leq0}
\mathcal{L}_{\va}^ \lambda  ( u_{\va, \lambda}  ) = F \quad  \text{ in } \Omega  \quad 
\text{ and } \quad u_{\va, \lm} -G \in H^2_0(\Omega; \R^d),
\end{equation}
with $F\in L^2(\Omega; \R^d)$ and $G\in H^2(\Omega; \R^d)$.
Let
\begin{equation}\label{w-e-11}
w_{\va, \lm}
=u_{\va, \lm} -v_{\va, \lm} -\va \chi^\lm (x/\va) \eta_\va S_\va (\nabla v_{\va, \lm}),
\end{equation}
where $\eta_\va$, $S_\va$ and $\chi^\lm$ are the same as in (\ref{w-e-1}).

\begin{lemma}\label{le1}
Let $\Omega$ be a bounded Lipschitz domain. 
Let $u_{\va, \lm} , v_{\va,\lm} $ be the weak solutions of \eqref{leq0} and \eqref{eqv0}, respectively, and $w_{\va, \lm}$ be given by \eqref{w-e-11}. Then
\begin{align}
\label{conver_es_w1_psi}
  \lm \va\| \Delta w_{\va, \lm}\|_{L^2(\Omega)} + \| \na w_{\va,\lm} \|_{L^2(\Omega)}
 &\le C\|  \nabla v_{\va, \lm}   \|_{L^{2}(\Omega_{5\varepsilon})} +C\varepsilon\| \nabla^{2}v_{\va, \lm} \|_{L^{2}(\Omega\setminus \Omega_{2\varepsilon})}
\end{align}
for $0<\va<1, $ where  $C$ depends only on $d$, $\nu_1$, $\nu_2$, and $\Omega$.
\end{lemma}

\begin{proof}

The proof is similar to that of (\ref{3.2-0}).
Let $(g)^\va =g(x/\va)$.
By direct calculations, we deduce that
\begin{align}
\label{conver_iden_L1}
\mathcal{L}^\lm_{\varepsilon} ( w_{\va, \lm})
&=\mathcal{L}^\lm_{0} (v_{\va, \lm}) 
+\lambda^2 \varepsilon^2 \Delta^{2}v_{\va, \lm}
-\mathcal{L}^\lm_{\varepsilon} (v_{\va, \lm} )
-\mathcal{L}^\lm_{\varepsilon}\big( \varepsilon (\chi^\lm)^\va  S_{\varepsilon} (\nabla v_{\va, \lm} )\eta_{\varepsilon} \big)\nonumber\\&
=-\textrm{div}\Big\{\widehat{A^\lm}\nabla v_{\va, \lm}
-A^\va \nabla v_{\va, \lm}
+\lambda^2 \varepsilon^3\Delta\nabla[  (\chi^\lm)^\va S_{\varepsilon} (\nabla v_{\va, \lm})\eta_{\varepsilon}]\nonumber\\
& \qquad\qquad\qquad
-\va A^\va \nabla[ (\chi^\lm)^\va
 S_{\varepsilon} (\nabla v_{\va, \lm} )\eta_{\varepsilon}]\Big\}\nonumber\\&
=-\textrm{div}\Big\{(A^\va -\widehat{A^\lm})[S_{\varepsilon} (\nabla v_{\va, \lm})\eta_{\varepsilon}-\nabla v_{\va, \lm} ] 
+(B^\lm)^\va S_{\varepsilon} (\nabla v_{\va, \lm})\eta_{\varepsilon}\nonumber\\
&\qquad
+\lambda^2 \varepsilon^3\Delta\nabla[S_{\varepsilon} (\nabla v_{\va, \lm})\eta_{\varepsilon}] (\chi^\lm)^\va
+\lambda^2 \varepsilon^2 (\nabla \chi^\lm)^\va \Delta[S_{\varepsilon} (\nabla v_{\va, \lm} )\eta_{\varepsilon}]\nonumber \\
& \qquad
+2\lambda^2 
\varepsilon^2 \nabla^{2}[S_{\varepsilon} (\nabla v_{\va, \lm})\eta_{\varepsilon}] (\nabla \chi^\lm)^\va  
+\lambda^2 \varepsilon (\Delta \chi^\lm)^\va \nabla[S_{\varepsilon} (\nabla v_{\va, \lm})\eta_{\varepsilon}]\nonumber \\
 & \qquad+2\lm^2 \varepsilon (\nabla^2 \chi^\lm)^\va \nabla[S_{\varepsilon} (\nabla v_{\va, \lm} )\eta_{\varepsilon}]
 -\varepsilon A^\va (\chi^\lm)^\va\nabla[S_{\varepsilon} (\nabla v_{\va, \lm})\eta_{\varepsilon}]\Big\},
\end{align} 
where $(B^\lm)^\va=B^\lm(x/\va)$  and $B^\lm$ is given by (\ref{B}).
Thus for any $\psi\in H^2_0(\Omega; \R^d)$,
\begin{align}
\label{rate w psi}
|\langle\mathcal{L}^\lm_{\varepsilon} ( w_{\varepsilon, \lm}),~\psi\rangle|&
\leq C\int_{\Omega}\big|[ \nabla v_{\va, \lm}   -S_{\varepsilon}( \nabla v_{\va, \lm}   )\eta_{\varepsilon}]\nabla \psi\big|\,dx
 +C\Big|\int_{\Omega}(B^\lm)^\va\eta_{\varepsilon}S_{\varepsilon}( \nabla v_{\va, \lm}   )\nabla \psi\,dx\Big|\nonumber\\
&\quad+C\lm^2\varepsilon^{3}\int_{\Omega}\big|
(\chi^\lm)^\va \nabla^3 [S_{\varepsilon}( \nabla v_{\va, \lm}   )\eta_{\varepsilon}] \nabla \psi\big| \,dx\nonumber\\
  &\quad+C\lm^2 \varepsilon^{2}\int_{\Omega} \big| (\nabla \chi^\lm)^\va\nabla^{2}[S_{\varepsilon}( \nabla v_{\va, \lm}   )\eta_{\varepsilon}]\nabla \psi \big| \,dx \nonumber\\
&\quad+C\lm^2\varepsilon \int_{\Omega} \big| (\nabla^2 \chi^\lm)^\va \nabla[S_{\varepsilon}( \nabla v_{\va, \lm}   )\eta_{\varepsilon}]\nabla\psi \big| \,dx\ \nonumber\\
&\quad+C \varepsilon\int_{\Omega} \big|(\chi^\lm)^\va\nabla[S_{\varepsilon}( \nabla v_{\va, \lm}   )\eta_{\varepsilon}]\nabla\psi\big|\,dx\nonumber\\
&
\doteq \mathcal{I}_1+\cdot\cdot\cdot+\mathcal{I}_{6}.
\end{align}
It is not hard  to see  that
$$
\aligned
\mathcal{I}_1&\leq
C\| \nabla v_{\va, \lm}   -S_{\varepsilon}( \nabla v_{\va, \lm}   )\|_{L^{2}(\Omega\setminus \Omega_{3\varepsilon})}\|\nabla \psi\|_{L^{2}(\Omega)}
+C\| \nabla v_{\va, \lm}  \|_{L^2(\Omega_{4\va})}  \| \nabla \psi \|_{L^2(\Omega_{4\va})}\\
&\le C \big\{ \|\nabla v_{\va, \lm}\|_{L^2(\Omega_{4\va})}
+ C \va \|\nabla^2 v_{\va, \lm} \|_{L^2(\Omega\setminus \Omega_{2\va})} \big\} 
\|\nabla \psi \|_{L^2(\Omega)}.
\endaligned
$$
 To handle $\mathcal{I}_2$, we use the matrix of flux correctors, as in the proof of Lemma \ref{lemma-3.1}, to obtain 
\begin{align*}
\mathcal{I}_2&= C \Big|\int_\Omega\varepsilon   \partial_{x_k}\big(\mathfrak{B}^\lm_{kij} (x/\va) \partial_{x_i}  \psi \big)S_{\varepsilon}( \partial_{x_j} v_{\va, \lm} )\eta_{\varepsilon} dx
\Big| \nonumber\\
 &\leq C\varepsilon\int_{\Omega}|\eta_{\varepsilon}\mathfrak{B}^\lm(x/\varepsilon) S_{\varepsilon}(\nabla^{2} v_{\va, \lm} )\nabla \psi|\,dx
 +C\varepsilon\int_{\Omega} |\mathfrak{B}^\lm(x/\varepsilon) S_{\varepsilon}( \nabla v_{\va, \lm}   )\nabla\eta_{\varepsilon}  \nabla \psi|\,dx\nonumber\\
 &\leq C\varepsilon\|\nabla \psi\|_{L^{2}(\Omega)}\| \nabla^{2} v_{\va, \lm} \|_{L^{2}(\Omega\setminus \Omega_{2\varepsilon})}
 +C \|\nabla \psi\|_{L^{2}(\Omega_{4\varepsilon})}\| \nabla v_{\va, \lm}    \|_{L^{2}(\Omega_{5\varepsilon})},
\end{align*}
where, for the last step, we have used \eqref{22}.

To bound $\mathcal{I}_3$, we  use the Cauchy  inequality, \eqref{echilm} and \eqref{22} to deduce that
\begin{align*}
\mathcal{I}_3&\leq C \lm^2 \va^3\| ( \chi^\lm)^\va \nabla^3 S_\va( \nabla v_{\va, \lm} ) \|_{L^2(\Omega\setminus\Omega_{3\va})} \|\na \psi \|_{L^2(\Omega)}\\
 & +C \lm^2 \Big\{\|S_\va(\na v_{\va, \lm} ) ( \chi^\lm)^\va \|_{L^2(\Omega_{4\va})} + \va \|S_\va(\na^2 v_{\va, \lm} ) ( \chi^\lm)^\va \|_{L^2(\Omega_{4\va})}\\
& \qquad\qquad\qquad
 + \va^2\|\nabla S_\va(\na^2 v_{\va, \lm} ) ( \chi^\lm)^\va \|_{L^2(\Omega_{4\va})}  \Big\}\|\na \psi \|_{L^2(\Omega_{4\va})}\\
 &\leq    C\varepsilon\| \nabla^{2}v_{\lm, \va} \|_{L^{2}(\Omega\setminus \Omega_{2\varepsilon})}\|\nabla \psi\|_{L^{2}(\Omega)}
 + C\| \nabla v_{\va, \lm}   \|_{L^{2}(\Omega_{5\varepsilon})}\|\nabla \psi\|_{L^{2}(\Omega_{4\varepsilon})}.
\end{align*}
Likewise,
\begin{align*}
\mathcal{I}_4+\mathcal{I}_5+\mathcal{I}_6
 \leq C\|  \nabla v_{\va, \lm}   \|_{L^{2}(\Omega_{5\varepsilon})}\|\nabla \psi\|_{L^{2}(\Omega)}
 +C\varepsilon\| \nabla^{2}v_{\va, \lm} \|_{L^{2}(\Omega\setminus \Omega_{2\varepsilon})}\|\nabla \psi\|_{L^{2}(\Omega)}.
\end{align*}
By taking the estimates on $\mathcal{I}_1, \dots, \mathcal{I}_{6}$ into \eqref{rate w psi}, it yields
$$
 |\langle\mathcal{L}^\lambda_{\varepsilon} (  w_{\va, \lm}) , \psi\rangle| 
 \le
 C\|  \nabla v_{\va, \lm}   \|_{L^{2}(\Omega_{5\varepsilon})}\|\nabla \psi\|_{L^{2}(\Omega)}
 +C\varepsilon\|  \nabla^{2}v_{\va, \lm} \|_{L^{2}(\Omega\setminus \Omega_{2\varepsilon})}\|\nabla \psi\|_{L^{2}(\Omega)},
$$
which gives \eqref{conver_es_w1_psi}  by choosing $\psi=w_{\va, \lm}\in H^2_0(\Omega; \R^d) $ and using the Cauchy  inequality.
\end{proof}

Now we are prepared to prove Theorem \ref{apth2}.

\begin{proof}[\textbf{Proof of Theorem \ref{apth2}}] By dilation and translation, it suffices to consider the case where
$r=1$ and $ z=0$.  Let $u_{\va, \lm}$  be a solution  of 
$
\mathcal{L}^\lm_\va  (u_{\va, \lm} ) = F
$
in $B_2$, 
and $v_{\va , \lm}$ the solution to the Dirichlet problem,
\begin{align*}
  \lm^2 \va^2 \Delta^2 v_{\va, \lm}  +
   \mathcal{L}^\lm_0 (  v_{\va , \lm}  ) = F  \quad   \text{ in } B_{3/2}   \quad \text{ and } \quad 
 v_{\va , \lm}  -  u_{\va, \lm} \in H^2_0(B_{3/2}; \R^d).
\end{align*}
Let $w_{\va, \lm}$ be defined by (\ref{w-e-11}).
We apply Lemma \ref{le1} with $\Omega=B_{3/2}$ to obtain 
\begin{equation}\label{lp-1}
\|\nabla w_{\va, \lm}\|_{L^2(B_{3/2} )}
\le C \|\nabla v_{\va, \lm} \|_{L^2( B_{3/2} \setminus B_{(3/2) -5\va})}
+ C \va \| \nabla^2 v_{\va,\lm} \|_{L^2( B_{(3/2) -2\va})}.
\end{equation}
Since $\widehat{A^\lm}$ is constant, we may apply (\ref{Ca-2}) to the function $\nabla v_{\va, \lm}$.
This gives
$$
\int_{B} |\nabla^2 v_{\va, \lm} |^2\, dx
\le \frac{C}{r^2} \int_{2B} |\nabla v_{\va, \lm} |^2\, dx
+ C \int_{2B} |F|^2\, dx,
$$
for any $2B=B(x_0, 2r)\subset B_2$. It follows that 
$$
\aligned
\int_{B_{(3/2) -2\va}} |\nabla^2 v_{\va, \lm}|^2\, dx
 & \le C \int_{B_{(3/2)  -\va}} \frac{|\nabla v_{\va, \lm} (x)|^2}{ [\delta (x) ]^2}\, dx
+ C \int_{B_{(3/2) -\va}} |F|^2\, dx\\
& \le C_q  \va ^{-1-\frac{2}{q}}  \left(\int_{B_{3/2}} |\nabla v_{\va, \lm}|^q \, dx \right)^{2/q}
+ C \int_{B_2} |F|^2\, dx,
\endaligned
$$
where $\delta (x)=\text{\rm dist} (x, \partial B_{3/2} )$, $q>2$ and we have used H\"older's inequality for the last step.
In view of (\ref{lp-1}) we deduce that for any $q>2$,
\begin{equation}\label{lp-3}
\|\nabla w_{\va, \lm} \|_{L^2(B_{3/2} )}
\le C \va^{\frac12 -\frac{1}{q}} \|\nabla v_{\va, \lm} \|_{L^q(B_{3/2} )} +C \va  \|F\|_{L^2(B_2)}.
\end{equation}

Next, we observe that $ u_{\va, \lm} -v_{\va,  \lm} \in H^2_0(B_{3/2} )$ and 
\begin{equation}\label{d-eq}
\lm^2 \va^2 \Delta^2 (u_{\va, \lm} -v_{\va, \lm})
-\text{\rm div} \big( \widehat{A^\lambda}
\nabla ( u_{\va, \lm} -v_{\va, \lm} ) \big)
=\text{\rm div} \big( ( A(x/\va) - \widehat{A^\lm} ) \nabla u_{\va, \lm} \big)
\end{equation}
in $B_{3/2}$.
By energy estimates this gives (\ref{s-ap-0}) with $r=1$.
It follows by Theorem \ref{Meyers-thm}  that there exist  some $q>2$ and $C>0$, depending only on $d$, $\nu_1$ and $\nu_2$, such that
$$
\int_{B_{3/2} } 
|\nabla ( u_{\va, \lm} - v_{\va, \lm})|^q\, dx
\le C \int_{B_{3/2} } | \nabla u_{\va, \lm} |^q\, dx.
$$
As a result, there exists some $q>2$ such that
\begin{equation}\label{lp-4}
\|\nabla w_{\va, \lm} \|_{L^2(B_{3/2})}
\le C \va^{\frac12 -\frac{1}{q}} \|\nabla u_{\va, \lm} \|_{L^q(B_2)}
+ C\va  \| F\|_{L^2(B_2)}.
\end{equation}

Note that for $x\in B_1$,
$$
\nabla w_{\va, \lm} =\nabla u_{\va, \lm} -\nabla v_{\va, \lm}
- ( \nabla \chi^\lm )^\va S_\va (\nabla v_{\va, \lm})
-\va ( \chi^\lm)^\va   S_\va (\nabla^2 v_{\va, \lm }).
$$
It follows from (\ref{lp-4})  that 
\begin{align} \label{apth201}
 &\| \na  u_{\va, \lm} - \na v_{\va, \lm} - (\na\chi^\lm)^\va   \na v_{\va, \lm} \|_{L^2(B_1 )}\nonumber\\
&\leq 
 C \va^{\frac12 -\frac{1}{q}} \|\nabla u_{\va, \lm} \|_{L^q(B_2)}
+ C\va  \| F\|_{L^2(B_2)}\nonumber \\
& \quad
+\|  (\nabla\chi^\lm )^\va \big( \nabla v_{\va, \lm}
-S_\va (\nabla v_{\va, \lm}) \big) \|_{L^2(B_1)}
+ \va \| (\chi^\lm)^\va S_\va (\nabla^2 v_{\va, \lm} ) \|_{L^2(B_1)}.
\end{align}
By (\ref{22}), the  last term in the right-hand side of (\ref{apth201}) is bounded by
$$
C  \va \| \nabla^2 v_{\va, \lm} \|_{L^2(B_{5/4}) }
\le C \va^{\frac12 -\frac{1}{q}} \|\nabla u_{\va, \lm} \|_{L^q(B_2)}
+ C\va  \| F\|_{L^2(B_2)}.
$$
To handle the third term in the right-hand side of  (\ref{apth201}), we use the $C^{1, \sigma}$ estimate
for the operator $\lm^2 \va^2 \Delta^2 -\text{\rm div} (\widehat{A^\lm} \nabla )$ to obtain 
\begin{equation}\label{C-1-a-5}
\|\nabla v_{\va, \lm}\|_{C^{0, \sigma } (B_{5/4} )}
\le C \| \nabla v_{\va, \lm }\|_{L^2(B_{3/2})}
+ C \| F\|_{L^p(B_{3/2})},
\end{equation}
where $0<\sigma < 1-\frac{d}{p}$.
It follows that
\begin{align}\label{papth1011}
\|  (\nabla\chi^\lm )^\va \big( \nabla v_{\va, \lm}
-S_\va (\nabla v_{\va, \lm}) \big) \|_{L^2(B_1)}
&\leq C   \| ( \na\chi^\lm) ^\va \|_{L^2(B_1)} 
\|  \na v_{\va , \lm}  -S_\va(\na v_{\va, \lm} )\|_{L^\infty(B_1)} 
\nonumber\\
&\leq C \va^{\sigma } \|\na  v_{\va, \lm } \|_{C^{0,\sigma }(B_{5/4} )}\nonumber\\
&\le C \va^\sigma \big\{ 
\| \nabla v_{\va, \lm} \|_{L^2(B_{3/2})} 
+\|F\|_{L^p(B_2)} \big\}\nonumber \\
&\leq C  \va^\sigma  \big\{ \|\na u_{\va , \lm} \|_{L^2(B_{3/2})} +\|F\|_{L^p(B_2)}\big\}.
\end{align}
In summary, we have proved that if $0<\sigma < \min (\frac{1}{2} -\frac{1}{q}, 1-\frac{d}{p})$, then 
\begin{equation}\label{lp-11}
\| \na  u_{\va, \lm} - \na v_{\va, \lm} - (\na\chi^\lm)^\va   \na v_{\va, \lm} \|_{L^2(B_1 )}
\le 
C  \va^\sigma  \big\{ \|\na u_{\va , \lm} \|_{L^q(B_{3/2} )} +\|F\|_{L^p(B_2)}\big\},
\end{equation}
where $2<q<\bar{q}$ and $\bar{q}$ depends only on $d$, $\nu_1$ and $\nu_2$.

Finally, we use the reverse H\"older estimate \eqref{RH-1} to obtain 
\begin{equation}
\| \nabla u_{\va, \lm } \|_{L^q(B_{3/2})}
\le C \big\{ \|\nabla u_{\va, \lm} \|_{L^2(B_2)}
+ \| F\|_{L^2(B_2)} \big\},
\end{equation}
where $q>2$ and $C$ depends only on $d$, $\nu_1$ and $\nu_2$.
This, together with (\ref{lp-11}), gives (\ref{apth2re1}) with $r=1$.
\end{proof}


\section{Large-scale  $C^{1,\alpha}$ estimates}\label{section-5}

Recall that $P_j^\beta (x)= x_j (0, \dots, 1, \dots, 0)$ with $1$ in the $\beta^{th}$ position.
Let
 \begin{equation}\label{plm}
 \aligned
\mathcal{H}_{1,\va}^\lm= &  \Big\{ h(x): \  h(x)=b+E_j^\beta  (P_j^\beta (x)  +\va \chi_j^{\lm, \beta}
(x/\va)) \\
&\qquad\qquad
 \text{ for some }  b\in \R^d \text{ and }   E=(E_j^\beta) \in  \R^{d\times d }\Big\}.
 \endaligned
\end{equation}

\begin{theorem}\label{c1al}
Assume that $A$ satisfies \eqref{econ} and \eqref{pcon}.
 Let $u_{\varepsilon, \lm} \in H^1(B_R; \R^d )$ be a solution of  $\mathcal{L}^\lm_\varepsilon ( u_{\varepsilon, \lm}) =F$
  in $B_R=B(x_0, R)$, where  $R>\va$ and $F\in L^p(B_R;\R^d )$ for some $p>d$. 
  Then for any $\va \leq r< R$ and  $0<\alpha<1-\frac{d}{p} ,$
\begin{align}\label{c1alre1}
\inf_{h\in \mathcal{H}_{1,\va}^\lm  }\left(\fint_{B_r} |\na u_{\va, \lm} 
-\na h |^2\right)^{1/2}\leq C \left(\frac{r}{R}\right)^{\alpha} \left\{\left(\fint_{B_R}
| \na u_{\va, \lm} |^2\right)^{1/2} +
R\left(\fint_{B_R} |F|^p\right)^{1/p}\right\},
\end{align}
where $C$ depends only on $d$, $\nu_1$, $\nu_2$, $p$,  and $\alpha$.
\end{theorem}

\begin{proof} 

By translation and dilation, we may assume that $x_0=0$ and $ R=2.$ 
We also assume that $\va<r<(1/8)$, as the estimate \eqref{c1alre1} is trivial for $r\geq(1/8)$.
 Let $v_{\va, \lm} $ be the weak solution of 
 $\va^2 \lm^2 \Delta^2 v_{\va, \lm} + \mathcal{L}_0^\lm (v_{\va, \lm} )  =F$ given by Theorem \ref{apth2}.
Let $\va<tr<r<1$, where $0<t<(1/4)$ is to be determined, and 
$$
\overline{h}=\na v_{\va, \lm} (0) (P(x) +\va\chi^\lm(x/\va)),
$$ 
where $P=(P_j^\beta(x))$.
We obtain 
\begin{align}\label{pc1a0}
 &\left(\fint_{B_{tr}}|\na u_{\va, \lm} -\na \overline{h}|^2\right)^{1/2} + tr\left(\fint_{B_{tr}}|F|^p\right)^{1/p}\nonumber\\
  &\leq  \left(\fint_{B_{tr}}|\na u_{\va, \lm} -\na v_{\va, \lm} -( \na \chi^\lm)^\va \na v_{\va, \lm} |^2\right)^{1/2}\nonumber\\
 &\quad+ \left(\fint_{B_{tr}}|\na v_{\va, \lm}  + ( \na \chi^\lm)^\va  \na v_{\va, \lm} -\na \overline{h}|^2\right)^{1/2}+C t^{1-d/p}r\left(\fint_{B_{2r}}|F|^p\right)^{1/p}.
\end{align}
Denote the first two terms in the right-hand side  of \eqref{pc1a0} by $(\ref{pc1a0})_1, (\ref{pc1a0})_2$. Thanks to Theorem \ref{apth2},
\begin{align}\label{pc1a1}
(\ref{pc1a0})_1&\leq C t^{-d/2}\left(\fint_{B_{r}}|\na u_{\va, \lm} -\na v_{\va, \lm} 
-( \na \chi^\lm)^\va  \na v_{\va, \lm} |^2\right)^{1/2}\nonumber\\
&\leq C t^{-d/2}  \left(\frac{\va}{r}\right)^{^\sigma} \left\{ \left(\fint_{B_{2r}}|\na u_{\va, \lm } |^2\right)^{1/2}+r \left(\fint_{B_{2r}}|F|^p\right)^{1/p}\right\}.
\end{align}
On the other hand, by the $C^{1,\alpha} $ estimate of $v_{\va, \lm} $,
\begin{align}\label{pc1a2}
  (\ref{pc1a0})_2&\le  \left(\fint_{B_{tr}}|\na v_{\va, \lm}  -\na v_{\va, \lm} (0)|^2\right)^{1/2} 
  +\left(\fint_{B_{tr} } | (\na \chi^\lm)^\va  [\na v_{\va, \lm} -\na v_{\va, \lm} (0)]|^2\right)^{1/2} \nonumber\\
 &\leq C(tr) ^{\gamma }  \| \nabla v_{\va, \lm} \|_{C^{0, \gamma }(B_{tr})}\nonumber\\
 &\leq C t ^{\gamma  } \left\{ \left(\fint_{B_r} | \na v_{\va, \lm} |^2 \right)^{1/2}+r  \left(\fint_{B_r} |F|^p\right)^{1/p} \right\}\nonumber\\
 &\leq C t  ^{\gamma  } \left\{  \left(\fint_{B_{2r}} |\na u_{\va, \lm} |^2 \right)^{1/2}+r  \left(\fint_{B_{2r}} |F|^p\right)^{1/p} \right\},
\end{align}
where  $0<\gamma < 1-\frac{d}{p}$ and we have used (\ref{s-ap-0}) for the last inequality.

Taking \eqref{pc1a1} and \eqref{pc1a2} into \eqref{pc1a0} and using the fact $ \mathcal{L}_\va^\lm ( h) =0$ for any $h\in \mathcal{H}_{1,\va}^\lm$, we derive that
\begin{align*}
&\inf_{h\in \mathcal{H}_{1,\va}^\lm }  \left\{\frac{1}{(rt)^\alpha}\left(\fint_{B_{tr}}|\na u_{\va, \lm} -\na h |^2\right)^{1/2}   + tr\left(\fint_{B_{tr}}|F|^p\right)^{1/p} \right\}\nonumber\\
&\leq C  \inf_{h\in \mathcal{H}_{1,\va}^\lm } \left\{t^{-d/2-\al}  \left(\frac{\va}{r}\right)^{^\sigma} + t^{\gamma -\alpha} \right\} \nonumber\\
 &\quad \times \frac{1}{(2r)^{\alpha}}\left\{ \left(\fint_{B_{2r}}|\na u_{\va, \lm}-\na h  |^2\right)^{1/2}+r \left(\fint_{B_{2r}}|F|^2\right)^{1/2}\right\}.
\end{align*}
For any $0<\alpha<1-\frac{d}{p} $, we first choose  $\gamma \in (\alpha, 1-\frac{d}{p}) $ and
then $t>0$ so small that $C t^{\gamma -\al}\leq 1/4.$
As a result,   if $r\geq N_0 \va$, where $N_0>1$ is so large  that
$$
 C t^{-d/2-\al}  \Big(\frac{\va}{r}\Big)^{^\sigma} \leq 1/4,
 $$
 then 
\begin{align*}
&\inf_{h\in \mathcal{H}_{1,\va}^\lm }  \left\{\frac{1}{(tr)^\alpha}\left(\fint_{B_{tr}}|\na u_{\va, \lm} -\na h |^2\right)^{1/2}   + tr\left(\fint_{B_{tr}}|F|^p\right)^{1/p} \right\}\nonumber\\
&
\leq  \frac{1}{2}
\inf_{h\in \mathcal{H}_{1,\va}^\lm }  \left\{ \frac{1}{(2r)^\alpha}\left(\fint_{B_{2r}}|\na u_{\va, \lm} -\na h |^2\right)^{1/2} +r \left(\fint_{B_{2r}}|F|^p\right)^{1/p} \right\}.
\end{align*} 
 By iteration, this implies that
\begin{align}\label{pc1a4}
&\inf_{h\in \mathcal{H}_{1,\va}^\lm }  \left\{\frac{1}{(tr)^\alpha}\left(\fint_{B_{tr}}|\na u_{\va, \lm} -\na h |^2\right)^{1/2}   + tr\left(\fint_{B_{tr}}|F|^p\right)^{1/p} \right\}\nonumber\\
&\leq \inf_{h\in \mathcal{H}_{1,\va}^\lm }  \left \{ \left(\fint_{B_{2}}|\na u_{\va, \lm} -\na h |^2\right)^{1/2} +\left(\fint_{B_{2}}|F|^p\right)^{1/p} \right\}
\end{align}
for any $r\geq N_0 \va.$ 
The case $\va\le r <  N_0 \va$ follows easily from the case $r=N_0\va$.
\end{proof}

As a corollary, we obtain a Liouville theorem for the operator $\mathcal{L}^\lm$.

\begin{theorem}
Suppose $A$ satisfies conditions \eqref{econ} and \eqref{pcon}.
Let $u\in H^2_{loc}(\R^d; \R^d)$ be a weak solution of
$$
\lm^2 \Delta^2 u -\text{\rm div}(A\nabla u) =0 \quad \text{ in } \R^d.
$$
Suppose that there exist $C>0$ and $\sigma\in (0, 1)$ such that
$$
\left(\fint_{B(0, R)} |u|^2 \right)^{1/2}
\le C R^{1+\sigma}
$$
for all $R>1$. Then there exist $b\in \R^d$ and $E=(E_j^\beta)\in \R^{d\times d}$ such that
$$
u(x) = b + E_j^\beta ( P_j^\beta + \chi^{\lm, \beta} _j (x))  \quad \text{ in } \R^d.
$$
\end{theorem}

\begin{proof}
This follows readily from Theorem \ref{c1al} with $\va=1$ and $F=0$.
\end{proof}


\section{Proof of Theorems \ref{lipth} and \ref{main-thm-3}}\label{section-6} 

\begin{theorem}\label{liplm}
Assume that $A$ satisfies \eqref{econ} and \eqref{pcon}.
Let $u_{\varepsilon, \lm} \in H^2(B_R; \R^d)$ be a solution of
 $\mathcal{L}^\lm_\varepsilon ( u_{\varepsilon, \lm} ) =F$ in $B_R$, 
 where  $F\in L^p(B_R; \R^d)$ for some $p>d$. Then for any $\va \leq r< R$,
\begin{align}\label{liplmre1}
\left(\fint_{B_r} |\na u_{\varepsilon, \lm}  |^2\right)^{1/2}\leq C   \left\{\left(\fint_{B_R}
| \na u_{\varepsilon,  \lm} |^2\right)^{1/2} +
R\left(\fint_{B_R} |F|^p\right)^{1/p}\right\},
\end{align}
 where $C$ depends only on $d$, $\nu_1$, $\nu_2$,  and $p$.
\end{theorem}

\begin{proof}
This follows from Theorem \ref{c1al}, as in the case of second-order elliptic equations \cite{fisher2016}.
We omit the details.
\end{proof}

\begin{proof}[\textbf{Proof of Theorem \ref{lipth}}]
Since $\mathcal{L}_\va =\mathcal{L}_\va^\lambda$ with $\lm =\kappa \va^{-1}$,
Theorem \ref{lipth} follows directly from Theorem \ref{liplm}.
\end{proof}

\begin{proof}[\textbf{Proof of Theorem \ref{main-thm-3}}]
By translation and dilation we may assume $r=1$ and $x_0=0$.
If $\va\ge (1/2)$, the H\"older norm of $A^\va=A(x/\va) $ is uniformly bounded.
The Lipschitz  estimate \eqref{T-lip}  follows directly from the $C^{1, \alpha}$  estimate in Theorem \ref{loc-thm}.
Consider the case  $0<\va<(1/2)$.
Let $u_\va\in H^2(B_1; \R^d)$ be a weak solution of $\mathcal{L}_\va (u_\va)=F$ in $B_1=B(0, 1)$, where $F\in L^p(B_1; \R^d)$ for some $p>d$.
Let $v(x)=\va u_\va (\va x)$. Then
$$
(\kappa \va^{-1})^2 \Delta^2 v -\text{\rm div} (A\nabla v) = F_\va,
$$
where $F_\va (x) =\va F(\va x)$.
By Theorem \ref{loc-thm}, 
$$
\aligned
|\nabla u_\va (0)|
& =|\nabla v(0)|\le
C \left\{ \left(\fint_{B_1} |\nabla v|^2 \right)^{1/2}
+ \left(\fint_{B_1} |F_\va|^p \right)^{1/p} \right\}\\
&=C \left\{ \left(\fint_{B_\va} |\nabla u_\va |^2 \right)^{1/2}
+  \va \left(\fint_{B_\va } |F|^p \right)^{1/p} \right\}\\
&\le  C \left\{ \left(\fint_{B_1} |\nabla u_\va |^2 \right)^{1/2}
+ \left(\fint_{B_1} |F|^p \right)^{1/p} \right\},
\endaligned
$$
where we have used  (\ref{L-L-0}) with $R=1$ for the last inequality.
\end{proof}

\bibliographystyle{amsplain}
\bibliography{Niu_Shen_2020.bbl}

\providecommand{\bysame}{\leavevmode\hbox to3em{\hrulefill}\thinspace}
\providecommand{\MR}{\relax\ifhmode\unskip\space\fi MR }
\providecommand{\MRhref}[2]{%
  \href{http://www.ams.org/mathscinet-getitem?mr=#1}{#2}
}
\providecommand{\href}[2]{#2}
\begin{thebibliography}{10}

\bibitem{armstrongar2016}
S.~N. Armstrong and J.~C. Mourrat, \emph{Lipschitz regularity for elliptic
  equations with random coefficients}, Arch. Ration. Mech. Anal. \textbf{219}
  (2016), no.~1, 255--348.

\bibitem{armstrongan2016}
S.~N. Armstrong and C.~K. Smart, \emph{Quantitative stochastic homogenization
  of convex integral functionals}, Ann. Sci. \'Ec. Norm. Sup\'er. (4)
  \textbf{49} (2016), no.~2, 423--481.

\bibitem{al87}
M.~Avellaneda and F.~Lin, \emph{Compactness methods in the theory of
  homogenization}, Comm. Pure Appl. Math. \textbf{40} (1987), no.~6, 803--847.

\bibitem{barton2016}
A.~Barton, \emph{Gradient estimates and the fundamental solution for
  higher-order elliptic systems with rough coefficients}, Manuscripta Math.
  \textbf{151} (2016), no.~3-4, 375--418.

\bibitem{lions1978}
A.~Bensoussan, J.-L. Lions, and G.~Papanicolaou, \emph{Asymptotic analysis for
  periodic structures}, AMS Chelsea Publishing, Providence, RI, 2011, Corrected
  reprint of the 1978 original.

\bibitem{fisher2016}
J.~Fischer and F.~Otto, \emph{A higher-order large-scale regularity theory for
  random elliptic operators}, Comm. Partial Differential Equations \textbf{41}
  (2016), no.~7, 1108--1148.

\bibitem{foseca2007}
I.~Fonseca, G.~Francfort, and G.~Leoni, \emph{Thin elastic films: the impact of
  higher order perturbations}, Quart. Appl. Math. \textbf{65} (2007), no.~1,
  69--98.

\bibitem{francfort1994}
G.~A. Francfort and S.~M\"{u}ller, \emph{Combined effects of homogenization and
  singular perturbations in elasticity}, J. Reine Angew. Math. \textbf{454}
  (1994), 1--35.

\bibitem{friedman1968}
A.~Friedman, \emph{Singular perturbations for partial differential equations},
  Arch. Rational Mech. Anal. \textbf{29} (1968), 289--303.

\bibitem{Gia-1983}
M.~Giaquinta, \emph{Multiple integrals in the calculus of variations and
  nonlinear elliptic systems}, Annals of Mathematics Studies, vol. 105,
  Princeton University Press, Princeton, NJ, 1983.

\bibitem{Gloria2015}
A.~Gloria, S.~Neukamm, and F.~Otto, \emph{Quantification of ergodicity in
  stochastic homogenization: optimal bounds via spectral gap on {G}lauber
  dynamics}, Invent. Math. \textbf{199} (2015), no.~2, 455--515.

\bibitem{gloriajems2017}
A.~Gloria and F.~Otto, \emph{Quantitative results on the corrector equation in
  stochastic homogenization}, J. Eur. Math. Soc. (JEMS) \textbf{19} (2017),
  no.~11, 3489--3548.

\bibitem{klsa1}
C.~E. Kenig, F.~Lin, and Z.~Shen, \emph{Homogenization of elliptic systems with
  {N}eumann boundary conditions}, J. Amer. Math. Soc. \textbf{26} (2013),
  no.~4, 901--937.

\bibitem{niuyuan2019}
W.~{Niu} and Y.~Yuan, \emph{Convergence rate in homogenization of elliptic
  systems with singular perturbations}, J. Math. Phys. \textbf{60} (2019),
  no.~11, 111509, 7.

\bibitem{schuss1976}
Z.~Schuss, \emph{Singular perturbations and the transition from thin plate to
  membrane}, Proc. Amer. Math. Soc. \textbf{58} (1976), 139--147.

\bibitem{shenapde2017}
Z.~Shen, \emph{Boundary estimates in elliptic homogenization}, Anal. PDE
  \textbf{10} (2017), no.~3, 653--694.

\bibitem{shennote2018}
\bysame, \emph{Periodic homogenization of elliptic systems}, Operator Theory:
  Advances and Applications, vol. 269, Birkh\"{a}user/Springer, Cham, 2018,
  Advances in Partial Differential Equations (Basel).

\bibitem{shenzhu2017}
Z.~Shen and J.~Zhuge, \emph{Convergence rates in periodic homogenization of
  systems of elasticity}, Proc. Amer. Math. Soc. \textbf{145} (2017), no.~3,
  1187--1202.

\bibitem{shu2000}
Y.~C. Shu, \emph{Heterogeneous thin films of martensitic materials}, Arch.
  Ration. Mech. Anal. \textbf{153} (2000), no.~1, 39--90.

\bibitem{suslinaD2013}
T.~A. Suslina, \emph{Homogenization of the {D}irichlet problem for elliptic
  systems: {$L^2$}-operator error estimates}, Mathematika \textbf{59} (2013),
  no.~2, 463--476.

\bibitem{zeppieri2016stochastic}
C.~I. Zeppieri, \emph{Stochastic homogenisation of singularly perturbed
  integral functionals}, Annali di Matematica Pura ed Applicata (1923-)
  \textbf{195} (2016), no.~6, 2183--2208.

\end{thebibliography}

\vspace{0.5cm}

\noindent Weisheng Niu \\
School of Mathematical Science, Anhui University, 
Hefei, 230601, CHINA.  \\
E-mail:niuwsh@ahu.edu.cn\\

\noindent Zhongwei Shen\\
Department of Mathematics, University of Kentucky,
Lexington, Kentucky 40506, USA.\\
E-mail: zshen2@uky.edu\\


\noindent\today

\end{document}